\documentclass[12pt, reqno]{amsart}
\usepackage{amsmath, amsthm, amscd, amsfonts, amssymb, graphicx, color, mathrsfs}
\usepackage[bookmarksnumbered, colorlinks, plainpages]{hyperref}
\usepackage[all]{xy}
\usepackage[utf8]{inputenc}
\usepackage{slashed}
\usepackage{cleveref}
\usepackage{mathtools}
\DeclareMathOperator{\spn}{span}

\newtheorem{theorem}{Theorem}[section]
\newtheorem{lemma}[theorem]{Lemma}

\newtheorem{corollary}[theorem]{Corollary}
\theoremstyle{definition}
\newtheorem{definition}[theorem]{Definition}

\theoremstyle{remark}
\newtheorem{remark}[theorem]{Remark}
\numberwithin{equation}{section}
\def\subg{{\nabla_{\mathbb{G}_{n+1}}}}
\def\subl{{\Delta_{\mathbb{G}_{n+1}}}}
\def\fr{{\frac{2n}{n+1}}}
\def\tn{{\tilde{N}}}
\def\aj{{ |x_1|^{\frac{n+1}{2}}+|x_2|^{\frac{n+1}{2}}+|x_j|^{\frac{n+1}{2(j-1)}}}}
\def\gp{{\mathbb{G}_{n+1}^{'}}}
\def\g{{\mathbb{G}_{n+1}}}
\def\one{\mbox{1\hspace{-4.25pt}\fontsize{12}{14.4}\selectfont\textrm{1}}}

\def\S{{\sum_{j=2}^{n}}}

\newcommand{\vertiii}[1]{{\left\vert\kern-0.25ex\left\vert\kern-0.25ex\left\vert #1 
    \right\vert\kern-0.25ex\right\vert\kern-0.25ex\right\vert}}
\allowdisplaybreaks

\begin{document}
	\setcounter{page}{1}
	
\title[$q$-Poincar{\'e} inequalities on Carnot Groups ]{$q$-Poincar{\'e} inequalities on Carnot Groups \\ with filiform type Lie algebra}
	
	\author[M. Chatzakou]{Marianna Chatzakou}
	\address{
		Marianna Chatzakou:
		\endgraf
		Department of Mathematics: Analysis Logic and Discrete Mathematics
		\endgraf
		Ghent University
		\endgraf
		Krijgslaan 281, Ghent, B 9000
        \endgraf
        Belgium
		\endgraf
		{\it E-mail address} {\rm marianna.chatzakou@ugent.be}
		}

	\author[S. Federico]{Serena Federico}
	\address{
		Serena Federico:
		\endgraf
			Department of Mathematics
			\endgraf
		University of Bologna
		\endgraf
		Piazza di Porta S. Donato 5, 40126, Bologna
		\endgraf
		Italy
		\endgraf
		{\it E-mail address} {\rm serena.federico2@unibo.it}}
		
		\author[B. Zegarlinski]{Boguslaw Zegarlinski}
	\address{
		Boguslaw Zegarlinski:
		\endgraf
		Université de Toulouse ; CNRS
		\endgraf
		UPS, F-31062 Toulouse Cedex 9
        \endgraf
        France
		\endgraf
		{\it E-mail address} {\rm b.zegarlinski@math.univ-toulouse.fr}
		}
		\endgraf

\thanks{Marianna Chatzakou was supported by the FWO Odysseus 1 grant G.0H94.18N: Analysis and Partial Differential Equations,
and by the Methusalem programme of the Ghent University Special Research Fund (BOF) (Grant
number 01M01021), and is a postdoctoral fellow of the Research Foundation – Flanders (FWO) under
the postdoctoral grant No 12B1223N. Serena Federico has received funding from the European Unions Horizon 2020 research and innovation programme under the Marie Sk{\l}odowska-Curie grant agreement No 838661}

	\begin{abstract} 
	In this paper we prove (global) $q$- Poincar{\'e} inequalities for probability measures on nilpotent Lie groups with filiform Lie algebra of any length. The probability measures under consideration have a density with respect to the Haar measure given as a function of a suitable homogeneous norm.
	\end{abstract} \maketitle
	
	\tableofcontents
	
	\section{Introduction}\label{Intro}
	The present paper is devoted to the investigation of global $q$-Poincaré inequalities on Carnot groups with a filiform  Lie algebra, where, in particular, a probability measure depending on a suitable homogeneous norm \footnote{The terminology homogeneous (with respect to the dilations) norm used in the setting of homogeneous groups, actually describes a quasi-norm.} on the group is used in place of the standard Haar measure on $G$. The class of Carnot groups treated in this paper will be called of {\it Engel} type, since, as we shall see below, they have the Engel group as prototype. An immediate consequence of these inequalities is given by the {\it spectral gap} for the Dirichlet operator defined by means of our probability measure.  
	
	\medspace
	
Coercive inequalities, including Poincaré inequalities, have a deep link with Gaussian bounds for the heat kernel. Classical bounds for the heat kernel in the nilpotent group setting, and its connection with Sobolev inequalities, were proved long time ago by Varopoulos, Saloff-Coste and Coulhon \cite{VSC92}  (see also references therein). More recently, gradient estimates for the heat kernel on the Heisenberg group have been achieved by Li in \cite{Li06}, who proved, as a consequence, the validity of coercive inequalities for the heat kernel semigroup. A simpler proof of the results in \cite{Li06} can be found in \cite{BBBC08}, where, once again thanks to the validity of gradient bounds, several functional inequalities for the heat kernel are derived. In the general setting of nilpotent Lie groups, gradient estimates for the heat kernel of a second order hypoelliptic operator can be found in \cite{M08}, were also the suitable connected functional inequalities are proved.

In the setting of stratified Lie groups, Poincar\'{e} inequalities on non-isotropic Sobolev spaces have been considered in the works by Lu et al. \cite{LLT15} \cite{CLW07}, \cite{Lu97}, \cite{Lu00}, \cite{LW00} and \cite{LW04}. Let us stress that all these results, which do not rely on the use of a probability measure, 
often provide $q-p$ type of Poincaré inequalities instead of $q-q$ type.

\medspace

A set of problems which has so far resisted a thorough exploration in this context is that of coercive inequalities involving sub-gradients and probability measures on nilpotent Lie groups.
In the present paper we shall focus on the latter problem, that is, more precisely, we will study global $q$-Poincaré inequalities by using a probability measure dictated by the group instead of the doubling property of the measure and heat kernel estimates.
This approach was recently developed by Hebisch and the third author in  \cite{HZ10} where they define a general strategy to study Poincar{\'e}, Log-Sobolev and other coercive inequalities on nilpotent Lie groups for a class of probability measures whose density (with respect to the Haar measure) is a function of the control (or Carnot-Carath\'{e}dory) distance. 
Specifically, the aforementioned technique is based on the use of quadratic form bounds providing a lower bound for a Dirichlet form involving the sub-gradient and a scalar potential, potential that is given as a function of the logarithm of the density of the probability measure. These bounds, called ``$U$-bounds'' in \cite{HZ10}, where $U$ stands for the potential, combined with the celebrated \textit{Poincar{\'e} inequality on balls} proved by Jerison in  \cite{Jer86} holding in the general nilpotent setting, allow to derive global  Poincar\'e inequalities on nilpotent Lie groups, provided, as we will see, that suitable gradient bounds for the potential are satisfied.

The quadratic form bounds in \cite{HZ10} are similar to the ones in the Euclidean setting that can be found in the works of Rosen \cite{Ros76} and Adams \cite{Ada79}. In \cite{HZ10} an application of this criterion is provided in the case of the Heisenberg group with probability measures depending on the control distance. In this setting, and for such measures, using the Poincar\'{e} and the Sobolev-Stein inequality (involving the sub-gradient and the Haar measure), also the Log-Sobolev inequality was proved. Additionally, it was shown that by replacing the control distance with any smooth distance (see \cite{HZ10}) the Log-Sobolev inequality fails to hold.

As pointed out in \cite{HZ10}, it is often convenient to use homogeneous norms different from the control distance. Indeed in \cite{Ing10} it was shown that in the case of the Heisenberg group, replacing the control distance with the Kaplan norm in the potential still allows for the Poincar{\'e} inequality to hold. Note also that the change of measure  also implies different spectral properties for the corresponding Dirichlet operators (see Remark 4.5.4 in \cite{Ing10}).
The difficulty in this consideration is that the quadratic form bounds do not provide a scalar potential growing to infinity in all directions, therefore an additional idea is necessary to treat the delicate region around the $Z$-axis where the sub-gradient of the Kaplan norm is small (unlike the sub-gradient of the control distance which satisfies the eikonal equation).  
 
Let us remark that the explicit knowledge of the Kaplan norm for $\mathbb{H}$-type groups is deeply used in \cite{Ing10}, while, in our case, due to the lack of such an explicit formula for the groups we will be considering (which are not of $\mathbb{H}$-type), 
 a different convenient homogeneous norm will be used.

	\medspace
	
	It is worth to mention that the global Poincar\'e inequalities studied here, have different applications with respect to their local counterparts. However, before stating more clearly the consequences these global inequalities have, let us first briefly recall some local results, and explain how they play an important role in our problem too.
	
	The first result concerning local Poincaré inequalities involving H\"{o}rmander's vector fields was proved by Jerison in \cite{Jer86}. Here the standard non-degenerate gradient is replaced by the possibly degenerate sub-gradient associated with the system of H\"{o}rmander's vector fields. Since we will be using the result in \cite{Jer86} later on, we state below Theorem 2.1 in \cite{Jer86}, which, in particular, holds in the general setting of nilpotent Lie groups.
\begin{theorem}
	\label{thm.jerison}
	Let $\mathbb G$ be any nilpotent Lie group, and let $r>0$, $x \in \mathbb G$. If $B_{r}(x)=\{y \in \mathbb G : d(x,y) \leq r \}$ is the ball of radius $r$ centered at $x$, then for all $p \in [1,\infty)$, there exists a constant $P_{0}(r)=P_{0}(r,p)$ such that for all $f \in C^{\infty}(B_r(x))$
	\[
	\int_{B_{r}(x)} |f(y)-f_{B_{r}(x)}|^{p}\,dy \leq P_{0}(r) \int_{B_{r}(x)} |\nabla_{\mathbb G}f(y)|^p\,dy\,,
	\]
	where $f_{B_{r}(x)}:= \frac{1}{|B_{r}(x)|}\int_{B_{r}(x)}f(y)\,dy$, and $dy$ denotes the Lebesgue measure.
\end{theorem} 

We remark that several other local results have been proved for H\"ormander's system of vector fields. 
Let us mention some of them, like, for instance, the ones by Franchi, Lu and Wheeden in \cite{FLW96}, and the ones by Garofalo and Nhieu in \cite{GN96}, both dealing with local $1$-Poincar\'e inequalities, giving, by duality, local Poincar\'e inequalities of  type $p-q$. Both considerations rely on the use of the doubling property of the measure and that they have applications in the resolution of the isoperimetric problem.

Other results in the nilpotent group setting are given by Ruzhansky and Surugan  in \cite{RS19} where different versions of local Poincar{\'e} inequalities are considered. 

Going back to Theorem \ref{thm.jerison}, we stress that this result is crucial in our global case too, since, specifically, we use a sort of localization technique where the local inequality is needed to control a suitable quantity in a bounded region.

\medspace

In what follows we will briefly describe our setting and state the main result of this paper.

The groups treated in this paper are of the form $\mathbb{G}_{n+1}=(\mathbb{R}^{n+1},\circ)$, $n\geq 3$,  where $\circ$ stands for the composition law
\begin{equation*}
x \circ y = 
\begin{pmatrix}
x_1+y_1 \\
x_2+y_2 \\
x_3+y_3+y_2x_1\\
x_4+y_4+y_3x_1+y_2\frac{x_{1}^{2}}{2!}\\
x_5+y_5+y_4x_1+y_3\frac{x_{1}^{2}}{2!}+x_2 \frac{x_{1}^{3}}{3!}\\
\vdots\\
x_{n+1}+y_{n+1}+y_nx_1+y_{n-1}\frac{x_{1}^{2}}{2!}+\cdots+y_3 \frac{x_{1}^{n-2}}{(n-2)!}+y_2\frac{x_{1}^{n-1}}{(n-1)!}\,
\end{pmatrix}\,.
\end{equation*}
These groups are stratified Lie groups with a filiform Lie algebra, and, observe, in the case $n=3$ we recover the so called {\it Engel} group.
We now consider the homogeneous (quasi-) norm $\tilde{N}$ on $\mathbb{G}_{n+1}$
\begin{equation*}
\tn(x):= \left(\|x\|^n+|x_{n+1}|\right)^{\frac{1}{n}}\,,
\end{equation*}
where $\|x\|^n:=\sum_{j=2}^{n}\left( \aj\right)^{\fr}$, and define on $\mathbb{G}_{n+1}$ the probability measure 
\begin{equation}
\label{prob.measure.carnot2}
 \mu_p(dx):=\frac{e^{-a\tn^p(x)}}{Z}dx,
\end{equation}
where $Z$ is a normalization constant.

The difficulty when dealing with general groups with our approach is to find an explicit homogeneous norm satisfying 
suitable bounds; see Lemma \ref{lemma.N,carnot} and the proof of Theorem \ref{Theorem 4.2} for the specific bounds needed. 
The necessity of these bounds motivates the choice of our setting, namely the one of Carnot groups with a filiform Lie algebra.
One might conjecture that the Kaplan norm could be such a choice, however its explicit formula is not at our disposal for groups other than the H-type groups. Under these motivating aspects, we chose to work in the setting of Carnot groups with a filiform Lie algebra, where the composition's law- and thus the associated vector fields'- formula allows for a quite natural choice of a homogeneous norm with the desired properties. 

Of course the same approach could be followed to prove the validity of $q$-global Poincar\'e inequalities in other more complicated settings, where, certainly, finding such a norm will (if it exists) be even more challenging.

With the previous definitions in mind we can now state the main result of the paper.

	\newtheorem*{thm.carnot}{Theorem \ref{thm.carnot}}
	\begin{thm.carnot}
	Let $\g$, $n \geq 3$ be a Carnot group of $n$-step. If $p \geq n$, then the measure $\mu_{p}$ as in \eqref{prob.measure.carnot2} satisfies a $q$-Poincar{\'e} inequality, i.e., there exists a constant $c_0$ such that 
	\begin{equation}
	\label{thm.carnot.stat2}
	\mu_{p} (|f-\mu_{p}(f)|^{q}) \leq c_0 \mu_{p}(|\subg f|^{q})\,,
	\end{equation}
	where $\frac{1}{p}+\frac{1}{q}=1$, for all functions $f$ for which the right hand side is finite.
	\end{thm.carnot}
	From Theorem \ref{thm.carnot} we can claim that \eqref{thm.carnot.stat2}, often called $q$-{\it spectral gap}, holds for  $q \leq 3/2$. However,  by Proposition 2.3 in \cite{BZ05}, we immediately get that if \eqref{thm.carnot.stat2} is true for $q$ then it is true for any $q'>q$ (and with the same unchanged measure $\mu_p$).
	Therefore from our result we obtain a 2-spectral gap (simply called {\it spectral gap}) for the Dirichlet operator defined with our probability measure, that is $\mathcal{L}=-\Delta_{\mathbb{G}}+\nabla_{\mathbb{G}} U\cdot \nabla_{\mathbb{G}}$ with $U$ being the probability density, and, consequently, we also get an exponential convergence to equilibrium in $L^2$ for the corresponding semigroup. 
	
	The spectral gap inequality is extremely important in order to derive information about the essential spectrum of the corresponding selfadjoint operator (see, for instance, \cite{GW06} and references therein).
	In particular, when a global $q$-Poincaré inequality of the form \eqref{thm.carnot.stat2} is true for $p\in (1,2)$, then no  spectral gap is valid and the operator defined by means of the homogeneous norm has no empty essential spectrum.

	\medspace
	
	We now conclude this introduction by giving the plan of the paper. 
	
	In Section 2 we provide a brief description of basics of analysis on groups of interest to us and make our setting more clear. 
	In Section 3 we prove global $q$- Poincar{\'e} inequalities on the  Engel group (Theorem \ref{thm.engel}) which represents the prototype for the more general case studied in Section 4. The strategy we use here extends the idea first used in \cite{Ing10} for  $\mathbb{H}$-type groups where a density depending on a homogeneous norm different than the control distance is used. This section also contains a lemma to get $q$-Poincarè inequalities for suitable measures different than $\mu_p$.
	We remark again that is not in general possible to pass from a $q$-spectral gap involving $\mu_p$ to the same gap involving a measure equivalent to $\mu_p$.
	Finally, the results of Section 3 are further generalised to nilpotent Lie groups with filiform Lie algebra of any length in Section 4.

	\section{Preliminaries on Carnot Groups}\label{Carnot}
	\textit{Carnot} groups are special cases of Carnot-Carath{\'e}odory spaces associated with a system of vector fields. In particular they are geodesic metric spaces initially introduced by Carath\'{e}odory in \cite{Car09} as a mathematical model of thermodynamics.\\
	\indent The setting of Carnot groups has many similarities with the Euclidean case (such as the geodesic distance, the presence of dilations and translations and the fact that they can naturally be equipped with an invariant measure called the Haar measure). Carnot groups become, therefore, highly interesting in many mathematical contexts. In particular Carnot groups appear mostly in harmonic analysis, in the study of hypoelliptic differential operators (cf. \cite{Ste93}, \cite{CDPT07}), as well as in the study of geometric measure theory (cf. \cite{Pan82}, \cite{Pan89}, \cite{Jer86}, \cite{LD13}, \cite{CL14}).\\
	
	\indent As for their geometric consideration, let us note that Carnot groups (or more generally stratified Lie groups) appear naturally in sub-Riemannian geometry (also called ``Carnot'' geometry). Roughly speaking Carnot groups can be served as the analogous of sub-Riemannian manifolds of the Euclidean vector spaces for Riemannian manifolds. More accurately, the tangent space at a point of a sub-Riemannian manifold can naturally be identified with a structure of a Carnot group (cf. \cite{Mit85}, \cite{BR96}). A direct approach to homogeneous Carnot groups can be found in \cite{Ste81}, \cite{VSC92}; see also \cite{HK00}.\\
	 
	 \indent Any Carnot group is naturally isomorphic to a homogeneous Lie group on $\mathbb{R}^n$ (see, for instance, \cite{BLU07}), i.e., Carnot groups can be realised as Lie groups with a global chart. Formally, they are defined as follows.
	 \begin{definition}\label{defn.Car.hom}
	 Let $\mathbb{G}=(\mathbb{R}^n, \circ)$ be a Lie group on $\mathbb{R}^n$ and let $\mathfrak{g}$ be the Lie algebra of $\mathbb{G}$. Then $\mathbb{G}$ is called a stratified group, or \textit{Carnot group}, if $\mathfrak{g}$ admits a vector space decomposition (stratification) of the form
	 \begin{equation}
	 \label{def.carnot}
	 \mathfrak{g}=\bigoplus_{j=1}^{r} V_j\,,\quad \text{such that}\quad 
	 \left\{
	 \begin{array}{l}
	 [V_1,V_{i-1}]= V_{i}\,,
	 \quad 2\leq i\leq r,\\
	 
	 [V_1,V_r]=\{0\},
	 \end{array}
	 \right.
	 \end{equation}
	 with $[V_i,V_j]$ denoting the Lie bracket of two arbitrary elements of the vector spaces $V_i$ and $V_j$.
	 \end{definition}

	 \begin{remark}
	As follows from Definition \ref{defn.Car.hom} a Carnot  group is nothing else than a Lie group whose Lie algebra $\mathfrak{g}$ is stratified (condition \eqref{def.carnot}).
	Any Carnot group admits at least one stratification, however, Definition \ref{def.carnot} is well posed since it does not depend on the  choice of the stratification (see, for instance, [Proposition 2.2.8 \cite{BLU07}]).
	Given a stratification $\mathfrak{g}=\oplus_{j=1}^{r}V_j$, where each $V_j$ consists of $n_j \neq 0$ elements of $\mathfrak{g}$, we write $x \in \mathbb{G}$ as
	 \[
	 x=(x^{(n_1)},\cdots,x^{(n_{j_r})})\,, \quad \text{where}\quad x^{(n_j)}\in \mathbb{R}^{n_j}\,,
	 \]
	 and the mapping $\delta_{\lambda}:\mathbb{R}^n \rightarrow \mathbb{R}^n$, $\lambda>0$, defined by
	 \[
	 \delta_{\lambda}(x):=(\lambda x^{(n_1)},\cdots,\lambda^{j_r}x^{(n_{j_{0})}})\,,
	 \]
	 is an automorphism of $\mathbb{G}$ for every $\lambda>0$;
see e.g. [Section 3.1.2 \cite{FR16}].
	 \end{remark}
	 There are several equivalent definitions of a Lie algebra $\mathfrak{g}$. Next we provide a characterisation of $\mathfrak{g}$ in the spirit of Definition \ref{defn.gen.op.subl-g}.\\
	 \indent Recall that a (smooth) vector field $X$ belongs to $\mathfrak{g}$ if and only if
	 \begin{equation}
	     \label{left-inv.v.f}
	     (XI)(\tau_{\alpha}(x))=\mathcal{J}_{\tau_{\alpha}(x)}(x) \cdot (XI)(x)\,,\quad \text{for all}\quad x \in \mathbb{G}\,,
	 \end{equation}
	 where $I$ stands for the identity map on $\mathbb{R}^n$, $(XI)(x)=(a_1(x),\ldots a_n(x))^t$ is the column vector of the components of $X$ at $x$,  and $\mathcal{J}_{\tau_{\alpha}(x)}$ denotes the Jacobian matrix at the point $x$ of the left-translation map $\tau_{\alpha}(x):= \alpha \circ x$, for some $\alpha \in \mathbb{G}$. In this case we say that $X$ is  a \textit{left-invariant vector field}.
	  In a similar way one can define the Lie algebra of a Lie group $\mathbb{G}$ in terms of right-invariant vector fields. In order to distinguish the Lie algebra generated by left-invariant vector fields from that generated by the right-invariant ones, we shall denote the latter by $\tilde{\mathfrak{g}}$.
	 For completeness we recall that a vector field $X\in T_e\mathbb{G}$, with $T_e\mathbb{G}$ being the tangent space of $\mathbb{G}$ at the neutral element $e\in \mathbb{G}$, is said to be  a \textit{right-invariant vector field} if it satisfies \eqref{left-inv.v.f} for $\tau_{\alpha}=\tilde{\tau}_{\alpha}:= x \circ \alpha$ (the right-translation map).\\
	 \indent The vector fields satisfying \eqref{left-inv.v.f} for $\alpha=0$, where $0$ is the identity element of $\mathbb{G}$, will be called the \textit{canonical left-invariant vector fields} in $\mathfrak{g}$. The \textit{canonical right-invariant vector fields} in $\tilde{\mathfrak{g}}$ are defined accordingly.
	 \begin{definition}
	 Let $\mathbb{G}$ be a stratified group. Let $(V_1,\ldots, V_r)$ be any stratification of the algebra of $\mathbb{G}$ as in Definition \ref{def.carnot}. Then we say that $\mathbb{G}$ has step (of nilpotency) $r$ and has $n_1$ generators, where $n_1:=\mathrm{dim}(V_1)$.
	 \end{definition}
	 
	 \begin{definition}
	 \label{defn.gen.op.subl-g}
	 Let $\mathbb{G}$ be a Carnot group, and let $\mathfrak{g}\,  (\tilde{\mathfrak{g}})$ be the corresponding Lie algebra. If $X_j$, $1 \leq j \leq n_1$, are the canonical left (right) invariant vector fields that generate $\mathfrak{g}\,  (\tilde{\mathfrak{g}})$, then the second order differential operator 
	 $$\Delta_{\mathbb{G}}=\sum_{j=1}^{n_1}X_{j}^{2}\,,
	 $$
 is called the \textit{canonical left (right) invariant sub-Laplacian} on $\mathbb{G}$, while the vector valued operator 
 $$\nabla_{\mathbb{G}}=(X_1,\cdots,X_{n_1})\,,$$
 is called the \textit{canonical left (right) invariant $\mathbb{G}$-gradient}. 
	 \end{definition}
	 Note that we have used the same notations for the left and the right-invariant sub-Laplacian and $\mathbb{G}$-gradient in Definition \ref{defn.gen.op.subl-g}. In order to avoid any confusion, we will always specify whether we are considering left or right-invariant vector fields.
	 
	 Below we give the explicit description of the Carnot groups we will be dealing with. As in [Section 4 \cite{BLU07}] we shall denote these Carnot groups of $n$-step by $\mathbb{G}_{n+1}=(\mathbb{R}^{n+1},\circ)$,  and by $\mathfrak{g}_{n+1}$ the corresponding Lie algebra of left-invariant vector fields. \\
	 
	 \indent Let $n \in \mathbb{N}$ be fixed. Let us consider the Lie algebra $\mathfrak{g}_{n+1}=\spn\{X_1,X_2,\cdots,X_{n+1}\}$
	 with commutator relations
	\begin{align*}
[X_i,X_j]&=0\,,\quad \quad\quad \quad 2\leq i,j \leq n+1\\
[X_1,X_j]&=X_{j+1}\,,\quad \quad 2 \leq j \leq n\\
[X_1,X_{n+1}]&=0\,.
\end{align*}
Then $\mathfrak{g}_{n+1}$ is an $(n+1)$-dimensional Lie algebra nilpotent of step $n$ that can be stratified as 
\[
\mathfrak{g}_{n+1}=\spn\{X_1,X_2\}\,\oplus\, \spn\{X_3\}\, \oplus\, \spn\{X_4\} \,\oplus\, \cdots \,\oplus\, \spn\{X_{n+1}\}\,.
\]
Observe that $\mathfrak{g}_{n+1}$ is stratified, implying that $\mathbb{G}_{n+1}$ is in particular a Carnot group. It is a routine to prove that the following (canonical, left-invariant) vector fields 
\[
X_1=\partial_{x_1}\,,\quad X_j= \sum_{k=j}^{n+1}\frac{x_{1}^{k-j}}{(k-j)!}\partial_{x_{k}}\,, \quad j=2,\cdots, n+1\,,
\]
satisfy the given commutator relations, and that, for any $\lambda>0$, the mapping
\[
\delta_{\lambda}(x_1,x_2,x_3,\cdots,x_{n+1})=(\lambda x_1,\lambda x_2, \lambda^2 x_3,\cdots, \lambda^n x_{n+1})\,,
\]
is an automorphism of $\mathbb{G}_{n+1}$. Finally we equip $\mathbb{G}_{n+1}=(\mathbb{R}^{n+1},\circ)$ with the composition law
\begin{equation*}
x \circ y = 
\begin{pmatrix}
x_1+y_1 \\
x_2+y_2 \\
x_3+y_3+y_2x_1\\
x_4+y_4+y_3x_1+y_2\frac{x_{1}^{2}}{2!}\\
x_5+y_5+y_4x_1+y_3\frac{x_{1}^{2}}{2!}+y_2 \frac{x_{1}^{3}}{3!}\\
\vdots\\
x_{n+1}+y_{n+1}+y_nx_1+y_{n-1}\frac{x_{1}^{2}}{2!}+\cdots+y_3 \frac{x_{1}^{n-2}}{(n-2)!}+y_2\frac{x_{1}^{n-1}}{(n-1)!}\,
\end{pmatrix}\,.
\end{equation*}
In what follows we shall consider groups of the form $\mathbb{G}_{n+1}$ for any $n\geq 3$ ($n=3$ and $n\geq 3$ in Section 3 and 4 respectively).
Note that, since in our case the topological dimension $\mathrm{dim}(\mathbb{G}_{n+1})=n+1\geq 4$ and the number of generators is $n$, i.e. $\mathrm{dim}(V_1)=n$, we have that $\mathbb{G}_{n+1}$ has a filiform Lie algebra (see Proposition 4.3.3 in \cite{BLU07}).

Let us also remark that ``being a Carnot group'' is invariant under isomorphisms of Lie groups, see e.g. [Proposition 2.2.10 \cite{BLU07}]. 
Additionally, Carnot groups of the same step are not necessarily isomorphic. Indeed, in \cite{BGR10} the authors study the heat kernel of two non-isomorphic Carnot groups of $3$-step, the so-called \textit{Engel group} $\mathcal{B}_4=(\mathbb{R}^4,\circ)$, and the so-called \textit{Cartan group} $\mathcal{B}_5=(\mathbb{R}^5,\ast)$. The Engel group $\mathcal{B}_4$ is $\mathbb{G}_{4}$, while the Cartan group $\mathcal{B}_5$ is not isomorphic to any $\mathbb{G}_{n+1}$.\\

\indent Finally, let us recall the notion of a homogeneous norm on a Carnot group.
\begin{definition}\label{def.hom.n}
We call \textit{homogeneous (quasi-)norm} on (the Carnot group) $\mathbb{G}$, every continuous \footnote{With respect to the Euclidean topology} mapping $N: \mathbb{G} \rightarrow [0,\infty)$ such that $N(x)>0$ if and only if $x \neq 0$, and 
\[
N(\delta_{\lambda}(x))=\lambda N(x)\,,\quad \text{for every}\quad \lambda>0\,,\quad x \in \mathbb{G}\,.
\]
\end{definition}

\indent The existence of geodesics in the setting of a 
Carnot group $\mathbb{G}$ (or even on more general settings, cf. \cite{HK00}) is well-known. Therefore, the \textit{control}, or \textit{Carnot-Carath\'{e}odory distance} (related to the generators of $\mathfrak{g}$) $d$ is well-defined on $\mathbb{G} \times \mathbb{G}$ giving rise to the metric $d_0$ defined by
\[
d_0(x):=d(x,0)\,,\quad x \in \mathbb{G}\,,
\]
where $d_0$ is a homogeneous norm on $\mathbb{G}$ (see e.g. [Theorem 5.2.8 \cite{BLU07}]), often simply denoted by $d$.

\section{$q$-Poincar{\'e}  inequality on the Engel group} 
	In this section we prove the $q$-Poincar{\'e} inequality in the setting of the Engel group $\mathcal{B}_4$ equipped with a probability measure.
	As already mentioned in the Introduction, the Engel group serves as a prototype for the more general filiform Carnot groups $\mathbb{G}_{n+1}$ described in the previous section. Let us emphasize that we shall work with right-invariant vector fields in this section, while in the general case described by $\mathbb{G}_{n+1}$  we shall use left-invariant vector fields instead. The reason of this choice is that, at least for the Engel group, this will show that global Poincaré inequalities hold independently of the use of left or right-invariant $\mathbb{G}$-gradients.
	\\
	\indent We start by defining a homogeneous (with respect to the dilations of the group) norm, denoted later on by $N$, on the Engel group $\mathcal{B}_4$, and, subsequently, a probability measure with density (with respect to the Haar measure) $U=e^{-aN^p}$ on $\mathcal{B}_4$, where $a>0$. In particular, we define:
	\begin{equation}\label{defin.norm.en}
N(x)= \left( \|x\|^3+|x_4|\right)^{\frac{1}{3}}\,,
\end{equation}
	where, for $x \in \mathcal{B}_4$, we define $\|x\|:=(x_{1}^{2}+x_{2}^{2}+|x_3|)^{\frac{1}{2}}$.\\
	
	In order to prove that the measure $\nu_p$ with density $U$(up to a normalization constant) satisfies $q$-Poincar\'{e} inequalities, we will need to prove some bounds for the norm $N$ in \eqref{defin.norm.en}. 
	The following lemma describes the behaviour of the norm $N$ under the action of the operators $\nabla_{\mathcal{B}_4}$ and $\Delta_{\mathcal{B}_4}$, which are the canonical right-invariant $\mathcal{B}_4$-gradient and canonical right-invariant sub-Laplacian on $\mathcal{B}_4$, respectively.
	\begin{lemma}\label{lemmaforN,engel}
Let $N$ be the norm on $\mathcal{B}_4$ given in \eqref{defin.norm.en}. Then, $N$ is smooth on $\mathcal{B}_{4}^{'}:=\mathcal{B}_4 \setminus \{ \{x:x_3=0\} \cup \{x:x_4=0\}\}$, and, in particular, for $x \in \mathcal{B}_{4}^{'}$, $x \neq 0$, we have the estimates
\begin{equation}\label{estim.subg}
 |\nabla_{\mathcal{B}_4} N(x)| \leq c_1 \frac{\|x\|^2}{N^2(x)}\,,
\end{equation}
and
\begin{equation}\label{estim.subl}
\Delta_{\mathcal{B}_4}N(x) \leq c_2 \frac{\|x\|}{N^2(x)}\,,
\end{equation}
for some positive constants $c_1$ and $c_2$.
\end{lemma}
We note that the expression \eqref{estim.subl} should be realised in the sense of distributions since the operation $\Delta_{\mathcal{B}_4}$ is not defined on the center of the group. In other words in Lemma \ref{lemmaforN,engel} we consider the restriction of $\Delta_{\mathcal{B}_4}$ on $\mathcal{B}'_4$, also denoted by $\Delta_{\mathcal{B}_4}|_{\mathcal{B}'_4}$. Since the latter is well defined in the usual sense, and since  $\Delta_{\mathcal{B}_4}\overset{\mathcal{D}'(\mathcal{B}'_4)}{=}\Delta_{\mathcal{B}_4}|_{\mathcal{B}'_4}$ (i.e. in the sense of distributions), it is enough for our purposes to consider $\Delta_{\mathcal{B}_4}|_{\mathcal{B}'_4}$. This will allow us to perform standard computations, that is with classical derivatives, and obtain our final result on $\mathcal{B}_4$ via an approximation argument.
\begin{proof}
The canonical right-invariant vector fields as calculated in \cite{C20} are given by 
\[
X_1=\partial_{x_1}-x_2\partial_{x_3}-x_3\partial_{x_4}\,,\quad \text{and}\quad X_j=\partial_{x_j}\,,j=2,3,4\,.
\]
Note also that
\[X_1(\|x\|^3+|x_4|)=\frac 3 2 \|x\|(2x_1+x_2sgn(x_3))-x_3 sgn(x_4),\]
and that
\[X_2(\|x\|^3+|x_4|)=3\|x\|x_2.\]
Hence for $x \in \mathcal{B}_{4}^{'}$, $x \neq 0$, we have
\[
X_1N(x)=\frac{1}{3N^2(x)}\left(\frac 3 2\|x\|(2x_1-sgn(x_3)x_2)-sgn(x_4)x_3\right)\,,\quad \text{and}\quad X_2N(x)=\frac{\|x\|x_2}{N^2(x)}\,.
\]
Therefore
\[
|X_1N(x)| \leq \frac{2\|x\|^2}{N^2(x)}\,,\quad \text{and}\quad |X_2N(x)| \leq \frac{\|x\|^2}{N^2(x)}\,,
\]
implying that 
\[
|\nabla_{\mathcal{B}_4}N(x)|^2=(X_1N(x))^2+(X_2N(x))^2 \leq  \frac{5\|x\|^4}{N^4(x)}\,.
\]
On the other hand $\Delta_{\mathcal{B}_4}$ can be estimated, on $\mathcal{B'}_4$, as
\begin{eqnarray*}\label{calc.delta}
\Delta_{\mathcal{B}_4}N(x)&=&(X_1)^2 N(x)+(X_2)^2 N(x)\nonumber\\
&=& - \frac{2}{9}\cdot\frac{ \left(3\cdot2^{-1}\|x\|\left(2x_1-x_2sgn(x_3)\right)-x_3 sgn(x_4)\right)^2+2(\|x\|x_2)^2}{N^5(x)}\\
&+&\frac{2^{-2}\|x\|^{-1}(2x_1-x_2sgn(x_3))^2+3x_2sgn(x_4)+2\|x\|+x_2^2\|x\|^{-1}}{N^2(x)}\\
&\leq& \frac{7\|x\|}{N^2(x)},
\end{eqnarray*}
which concludes the proof.

\end{proof}
Our choice of the norm $N$ allow us to equip the Engel group $\mathcal{B}_4$ with the probability measure 

\begin{equation}\label{prob.m.engel}
\nu_{p}(dx):= \frac{e^{-aN^{p}(x)}}{Z}dx\,,
\end{equation}
where $p \in (1,\infty)$, $a>0$, $dx$ is the Lebesgue measure on $\mathbb{R}^4$ \footnote{The Haar measure of a Carnot group $\mathbb{G}=(\mathbb{R}^n,\circ)$ coincides with the Lebesgue measure on $\mathbb{R}^n$.}, and $Z=\int e^{-a N^{p}(x)}\,dx$ is the normalisation constant. We then have the following result. 
\begin{theorem}\label{thm.engel} 
\label{Theorem 3.2.}
If $p \geq 3$, then the measure $\nu_{p}$ given by \eqref{prob.m.engel} satisfies the following $q$-Poincar{\'e} inequality, i.e., there exists a constant $c_0\in(0,\infty)$ such that 
\[
\nu_{p} (|f-\nu_{p}(f)|^{q} )\leq c_0 \nu_{p}(|\nabla_{\mathcal{B}_4} f|^{q})\,,
\]
 where $\frac{1}{p}+\frac{1}{q}=1$, for all functions $f$ for which the right hand side is finite.
\end{theorem}
\begin{remark}
 Note that we can obtain Poincar{\'e} inequality with $q < 2$ which is much stronger than and implies the ordinary Poincar{\'e} inequality in $L^2$ (see e.g. Proposition 2.3 in \cite{BZ05}). Thus as a consequence we get a spectral gap for the Dirichlet operator defined with our probability measures and hence an exponential convergence to equilibrium in $L^2$ for the corresponding semigroup. We also note that using the same perturbation techniques as in \cite{BZ05} we can include a class of probability measures which, besides $N^p$, contain terms with slower growth at infinity.
 \end{remark}

Before turning on to prove Theorem \ref{thm.engel} one needs to proceed using techniques similar to the $U$-bound as developed in \cite{HZ10}. In particular we make use of the following lemma in which  
for $x \in \mathcal{B}_4$, we set $\vertiii{x}=|x_2|$. 

\begin{lemma}\label{lemma,engel}
Let $p,q$ be as in Theorem \ref{thm.engel}. Then, for the probability measure $\nu_p$ as in \eqref{prob.m.engel}, there exist constants $C,D\in(0,\infty)$ such that 
\begin{equation}\label{lemma,eng,stat}
\nu_p (|f|^q N^{p-3} \vertiii{\cdot}^3) \leq C \nu_{p} (|\nabla_{\mathcal{B}_4} f |^{q})+D \nu_{p} (|f|^{q})\,,
\end{equation}
for any $f$ for which the right hand side is well defined.

\end{lemma}
\begin{proof}We start by splitting the set
 $\mathcal{B}_{4}^{'}$ as in Lemma \ref{lemmaforN,engel} into its connected components $C_j$; that is we write $\mathcal{B}_{4}^{'}=\cup_{j \in \mathcal{J}}C_j$, where $\mathcal{J}$ is a finite set of indices.

\indent Now, for some fixed $j \in \mathcal{J}$, we consider $f$ to be such that $0 \leq f \in C^{\infty}(C_j)$ and is compactly supported; for instance one can choose $f$ to be supported on the set \begin{equation}\label{Bj}
B_j=B_j(\tilde{x},r_j):=\{x \in C_j: d(\tilde{x},x) \leq r_j\}
\end{equation}
for some $\tilde{x} \in C_j$ and for some $r_j>0$  such that $B_j(\tilde{x},r_j) \subset C_j$, where $d$ is the Carnot-Carathéodory metric. \\
\indent Then, clearly, $fe^{-N^p}$ is a differentiable function on $C_j$, and an application of the Leibniz rule yields,
\[
e^{-aN^{p}}(\nabla_{\mathcal{B}_4}f)=\nabla_{\mathcal{B}_4}(fe^{-aN^p})+apfN^{p-1}(\nabla_{\mathcal{B}_4}N)e^{-aN^p}\,,
\]
 so that, by taking the inner product of the above quantity with $\frac{N^2}{\|\cdot\|^2}\nabla_{\mathcal{B}_4}N$ and integrating over $C_j$ with respect to $\nu_{p}$, one gets
\begin{eqnarray*}
\lefteqn{\frac{1}{Z}\int_{C_j} \frac{N^2(x)}{\|x\|^2} \nabla_{\mathcal{B}_4}N(x)  \cdot\nabla_{\mathcal{B}_4}f(x)e^{-aN^p(x)}\,dx}\nonumber\\
 & = &\frac{1}{Z} \int_{C_j} \frac{N^2(x)}{\|x\|^2} \nabla_{\mathcal{B}_4}N(x)\cdot \nabla_{\mathcal{B}_4}\left(f(x)e^{-aN^p(x)}\right)\,dx\nonumber\\
&+& \frac{ap}{Z} \int_{C_j} f(x) \frac{N^{p+1}(x)}{\|x\|^2} |\nabla_{\mathcal{B}_4}N(x)|^2e^{-aN^p(x)}\,dx\,.\nonumber\\
\end{eqnarray*}
\indent An application of the Cauchy–Schwartz inequality gives $\nabla_{\mathcal{B}_4}N(x)\cdot \nabla_{\mathcal{B}_4}f(x)\leq |\nabla_{\mathcal{B}_4}N(x)||\nabla_{\mathcal{B}_4}f(x)|$, so that by the above inequality one has
\begin{eqnarray}\label{firstineq}
\lefteqn{\frac{1}{Z}\int_{C_j} \frac{N^2(x)}{\|x\|^2} |\nabla_{\mathcal{B}_4}N(x)||\nabla_{\mathcal{B}_4}f(x)|e^{-aN^p(x)}\,dx}\nonumber\\
& \geq &\frac{1}{Z} \int_{C_j} \frac{N^2(x)}{\|x\|^2} \nabla_{\mathcal{B}_4}N(x)\cdot\nabla_{\mathcal{B}_4}\left(f(x)e^{-aN^p(x)}\right)\,dx\nonumber\\
&+& \frac{ap}{Z} \int_{C_j} f(x) \frac{N^{p+1}(x)}{\|x\|^2} |\nabla_{\mathcal{B}_4}N(x)|^2 e^{-aN^p(x)}\,dx\,.\nonumber\\
\end{eqnarray}
Notice that the first term of the right-hand side of \eqref{firstineq} can be treated by using integration by parts as follows
\begin{eqnarray}\label{integbyparts}
\frac{1}{Z}\int_{C_j} \frac{N^2(x)}{\|x\|^2} \nabla_{\mathcal{B}_4}N(x)\cdot\nabla_{\mathcal{B}_4}\left(f(x)e^{-aN^p(x)}\right)\,dx=\nonumber\\
-\frac{1}{Z}\int_{C_j} f(x) \nabla_{\mathcal{B}_4}\cdot \left( \frac{N^2(x)}{\|x\|^2}\nabla_{\mathcal{B}_4}N(x) \right)e^{-aN^p(x)}\,dx\,,
\end{eqnarray}
while, by Lemma \ref{lemmaforN,engel}, one can check that 
$
|\nabla_{\mathcal{B}_4}N|^2\geq |X_2N(x)|^2 \geq \frac{\|\cdot\|^2 \vertiii{\cdot}^2}{N^4}$, 
so that the combination of \eqref{firstineq}, \eqref{integbyparts} and \eqref{estim.subg} gives 
\begin{eqnarray}\label{secondineq}
\lefteqn{\frac{ap}{Z}\int_{C_j} f(x)N^{p-3}(x)\vertiii{x}^2 e^{-aN^p(x)}\,dx}\nonumber \\
& \leq & \frac{c_1}{Z}\int_{C_j}  |\nabla_{\mathcal{B}_4}f(x)| e^{-aN^p(x)}\,dx\nonumber\\
&+& \frac{1}{Z}\int_{C_j} f(x) \nabla_{\mathcal{B}_4} \cdot \left( \frac{N^2(x)}{\|x\|^2}\nabla_{\mathcal{B}_4}N(x) \right) e^{-aN^p(x)}\,dx\,.\nonumber\\
\end{eqnarray}
Further calculations show that for any $x \in \mathcal{B}_{4}^{'}$ one has:
\begin{eqnarray}
\label{cs2}
\nabla_{\mathcal{B}_4} \cdot \left( \frac{N^2(x)}{\|x\|^2}\nabla_{\mathcal{B}_4}N(x) \right) &=& \frac{2N(x) |\nabla_{\mathcal{B}_4}N(x)|^2}{\|x\|^2}+\frac{N^2(x)}{\|x\|^2} \Delta_{\mathcal{B}_4}N(x)\nonumber\\
&-&2\frac{N^2(x)}{\|x\|^3} \nabla_{\mathcal{B}_4}N(x) \cdot \nabla_{\mathcal{B}_4} \|x\|\nonumber\\
& \leq & 2c_{1}^{2} \frac{\|x\|^2}{N^3(x)}+(c_2+c) \frac{1}{\|x\|}\,,
\end{eqnarray}
where we have used \eqref{estim.subg} and \eqref{estim.subl}, and since $|\nabla_{\mathcal{B}_4}\|\cdot\||$ is bounded by a constant, the inner product $\frac{N^2(x)}{\|x\|^3}\nabla_{\mathcal{B}_4}N(x) \cdot \nabla_{\mathcal{B}_4}\|x\|$ has been estimated by
\[
|\nabla_{\mathcal{B}_4}N(x)|| \nabla_{\mathcal{B}_4}\|x\||\leq \frac{c}{\|x\|}\,,\quad \text{for some}\quad c>0,
\] by using the Cauchy-Schwartz inequality and \eqref{estim.subg}. Hence \eqref{secondineq} under the estimate \eqref{cs2} becomes 
\begin{eqnarray*}
ap \nu_{p}(fN^{p-3}\vertiii{\cdot}^{2}\one_{C_j}) & \leq & c_1\nu_{p}(|\nabla_{\mathcal{B}_4}f| \one_{C_j})+2c_{1}^{2} \nu_{p}\left(f\frac{\|\cdot\|^{2}}{N^3}\one_{C_j}\right)\\
&+& (c_2+c)\nu_{p}\left(\frac{f}{\|\cdot\|}\one_{C_j}\right)\,,
\end{eqnarray*}
and after replacing $f$ by $f\vertiii{\cdot}$ the last yields:
\begin{eqnarray}\label{beforeputq}
\lefteqn{ap \nu_{p}(f N^{p-3}\vertiii{\cdot}^3 \one_{C_j})}\nonumber\\
& \leq & c_1\nu_{p}(\vertiii{\cdot}|\nabla_{\mathcal{B}_4}f|\one_{C_j})+c_1\nu_{p}(f|\nabla_{\mathcal{B}_4}\vertiii{\cdot}| \one_{C_j})\nonumber\\
&+&2c_{1}^{2} \nu_{p} \left(f \frac{\|\cdot\|^2 \vertiii{\cdot}}{N^3}\one_{C_j} \right)+(c_2+c) \nu_{p} \left(f \frac{\vertiii{\cdot}}{\|\cdot\|}\one_{C_j} \right)\nonumber\\
& \leq & c_{1}^{'}\nu_{p}(\vertiii{\cdot}|\nabla_{\mathcal{B}_4}f| \one_{C_j}) + c_{2}^{'}\nu_{p}(f\one_{C_j})\,,
\end{eqnarray}
for some new constants $c_{1}^{'}, c_{2}^{'}$ since $|\nabla_{\mathcal{B}_4}|x_2||=|\nabla_{\mathcal{B}_4}x_2|=1$, and $\frac{\|\cdot\|^2\vertiii{\cdot}}{N^3},\frac{\vertiii{\cdot}}{\|\cdot\|}\leq 1$.

Now, replace $f$ by $f^q$, where $q$ is such that $\frac{1}{q}+\frac{1}{p}=1$, and $p$ is as in the statement, so that \eqref{beforeputq} becomes:
\begin{equation}
\label{afterputq}
ap \nu_{p}(f^q N^{p-3}\vertiii{\cdot}^3 \one_{C_j}) \leq  c_{1}^{'} q\nu_{p}(\vertiii{\cdot}f^{q-1}|\nabla_{\mathcal{B}_4}f| \one_{C_j})+c_{2}^{'} \nu_{p}(f^q \one_{C_j})\,.
\end{equation}

Now, an application of Young's inequality gives that, for any $\epsilon>0$ and for $q,p$ as before, we have
\[
qab \leq \frac{1}{\epsilon^{q-1}}a^q+\frac{q}{p}\epsilon b^p\,,\quad a,b \geq 0\,,
\]
so that, by choosing $a=|\nabla_{\mathcal{B}_4}f|$ and $b=\vertiii{\cdot}f^{q-1}$, we get
\[
q \vertiii{\cdot}f^{q-1}|\nabla_{\mathcal{B}_4}f| \leq \frac{1}{\epsilon^{q-1}}|\nabla_{\mathcal{B}_4}f|^q+\frac{q}{p} \epsilon \vertiii{\cdot}^pf^q\,,
\]
whereas, since $\vertiii{\cdot}^{p} \leq \vertiii{\cdot}^3 N^{p-3}$, using \eqref{afterputq}, we get
\[
(ap-c_{1}^{'}\frac{q}{p} \epsilon) \nu_{p}(f^{q} N^{p-3}\vertiii{\cdot}^{3} \one_{C_j})\leq \frac{c_{1}^{'}}{\epsilon^{q-1}}\nu_{p}(|\nabla_{\mathcal{B}_4}f|^q \one_{C_j})+c_{2}^{'} \nu_{p}(f^{q}\one_{C_j})\,,
\]
and by choosing $\epsilon$ such that $ap-c_{1}^{'}\frac{q}{p}\epsilon>0$, we conclude that
\begin{equation*}
\label{poin.withCj}
\nu_{p}(f^{q}N^{p-3}\vertiii{\cdot}^3\one_{C_j}) \leq C \nu_{p} ( |\nabla_{\mathcal{B}_4}f|^q \one_{C_j})+D \nu_{p}(f^q \one_{C_j})\,,
\end{equation*}
where 
\[
C=\frac{c_{1}^{'}}{\epsilon^{q-1}(ap-c_{1}^{'}\frac{q}{p}\epsilon)}\,,\quad \text{and}\quad D=\frac{c_{2}^{'}}{ap-c_{1}^{'}\frac{q}{p}\epsilon}\,,
\]
and the proof of \eqref{lemma,eng,stat} for smooth, non-negative, compactly supported functions on $C_j$ is completed.\\

\indent Now, to handle non-negative, non-smooth functions on the same domain $C_j$, one can use an approximation argument. In particular, for $f \in W^{1,q}(C_j)$, where $W^{1,q}(C_j)$ is the Euclidean Sobolev space and $q$ is as in the hypothesis, there exists a sequence $(f_n) \in C_{c}^{\infty}(\mathbb{R}^4)$ such that  $\left.f_n\right|_{C_j}$ satisfies  \eqref{lemma,eng,stat}, and $\left.f_n\right|_{C_j} \longrightarrow f$ in $W^{1,q}(C_j)$. Indeed, as in the Euclidean setting, one can choose $f_n=\zeta_n (\rho_n \ast f)$, with $\rho_n \in C_0^\infty(\mathbb{R}^4)$ being a regularizing sequence and $\zeta_n$ being a suitable cut-off function, so that $f_n:=\zeta_n(\rho_n \ast f) \in W^{1,q}(\mathbb{R}^4)$ has compact support (the sets $B_j$ as in \eqref{Bj} are compact with respect to the Euclidean topology, so we are allowed to use the standard approximation technique). \\
\indent We claim that such $(f_n)$ can approximate $f \in W^{1,q}(C_j)$ as in the left and the right-hand side of inequality \eqref{lemma,eng,stat}. Indeed, for the left hand side \eqref{lemma,eng,stat} we have, for any $n \in \mathbb{N}$,
\begin{eqnarray*}
\nu_{p}\left(|f-f_n|^q N^{p-3}\vertiii{\cdot}^3 \one_{C_j}\right)& = &\frac{1}{Z}\int_{C_j}|f(x)-f_n(x)|^q N(x)^{p-3}\vertiii{x}^3e^{-aN^p(x)}\,dx\\
 & \leq & \frac{C}{Z} \|(f-f_n)^q\one_{C_j}\|_{L^1(dx)}=\frac{C}{Z} \|(f-f_n)\one_{C_j}\|_{L^q(dx)}^{q} , 
\end{eqnarray*}
 where we used H\"{o}lder's inequality together with  $\lim\limits_{|x|\rightarrow \infty}N^{p-3}(x)\vertiii{x}^3e^{-aN^p(x)}=0$. Thus, by the inequality above, we have
 \begin{equation}
     \label{lhs.app}
     \nu_{p}(f_{n}^{q}N^{p-3}\vertiii{\cdot}^3\one_{C_j})\longrightarrow \nu_{p}(f^{q}N^{p-3}\vertiii{\cdot}^3\one_{C_j})\,\quad \text{as}\quad n \longrightarrow \infty\,.
 \end{equation}

\indent On the other hand, for the right-hand side of \eqref{lemma,eng,stat} we proceed as follows. First observe that for any $n \in \mathbb{N}$ by H\"older's inequality we have
\begin{eqnarray}
\label{appr.rhs.poi}
	\|[x_i(\partial_{x_j}f-\partial_{x_j}f_n)]^2 (e^{-aN^p})^{\frac{2}{q}}\one_{C_j}\|_{L^{\frac{q}{2}}(dx)}^{\frac{q}{2}}& =& \||\partial_{x_j}f-\partial_{x_j}f_n|^q |x_{j}|^{q}e^{-aN^{p}}\one_{C_j}\|_{L^{1}(dx)}\nonumber\\
	& \leq & C \|(\partial_{x_j}f-\partial_{x_j}f_n)\one_{C_j}\|^{q}_{L^q(dx)}\,,\end{eqnarray}
for every $1 \leq i,j \leq 4$, since $\lim\limits_{|x|\rightarrow \infty}|x_{j}|^{q}e^{-aN^{p}(x)}=0$. Direct calculations show that 
\begin{equation}
\label{appr.rhs.poi2}
	\nu_{p}(|\nabla_{\mathcal{B}_4}(f_n-f)|^q \one_{C_j})\leq \frac{1}{Z}\sum_{j=1}^{2} \| [X_j(f_n-f)]^2 (e^{-aN^p})^{\frac{2}{q}}\one_{C_j}\|_{L^{\frac{q}{2}}(dx)}^{\frac{q}{2}}\,,
\end{equation}
so that since each $X_k$, $k=1,2$ is expressed as the sum of terms of the form $x_i\partial_{x_j}$, inequalities \eqref{appr.rhs.poi} and \eqref{appr.rhs.poi2} imply
\begin{equation}
    \label{rhs.app}
    \nu_{p}(|\nabla_{\mathcal{B}_4}f_n|^q \one_{C_j})\longrightarrow  \nu_{p}(|\nabla_{\mathcal{B}_4}f|^q \one_{C_j})\,,\quad \text{as}\quad n \longrightarrow \infty\,.
\end{equation}
Finally, combining \eqref{lhs.app} with \eqref{rhs.app}, one shows that \eqref{lemma,eng,stat} holds true for $f \in W^{1,q}(C_j)$, $f \geq 0$.

\indent To extend the domain of $f$ it is enough to observe that, since $\nu_p(\mathcal{B}_4)=\nu_p(\mathcal{B}_{4}^{'})$, we can write $f$ as $f=\sum_{j \in \mathcal{J}} \left.f\right|_{C_j}+\left.f\right|_{\mathcal{B}_4 \setminus \mathcal{B}_{4}^{'}}$.\\
\indent Finally, to handle $f$ of arbitrary sign one can replace $f$ by $|f|$, and use the equality $\nabla_{\mathcal{B}_4}|f|=sgn(f)\nabla_{\mathcal{B}_4}f$. This completes the proof.   
\end{proof}

We note that, the so called $U$-bounds mentioned above are exactly estimates of the form 
\begin{equation}\label{Ubound}
\int |f|^q g(d)\,d\mu \leq A_q \int |\nabla f|^q\,d\mu+B_q \int |f|^q\,d\mu\,,
\end{equation}
where  $\mu=e^{-U(d)}dx$  is a probability measure, and $d$ is a homogeneous norm.\\
\indent In Theorem 2.1 in \cite{HZ10} the authors proved that on any finite dimensional space with probability measure which does not necessarily satisfy the doubling property,  if $\mu$ satisfies \eqref{Ubound} for any function $g(d)$ increasing to infinity when $d\to\infty$,  with $\mu$ being absolutely continuous with respect to some measure $\lambda$ satisfying the $q$-Poincar\'{e} inequality on balls, then $\mu$ satisfies the $q$-Poincar\'{e} inequality, provided that for any $L>0$ there exists some $R=R(L) \in (0,\infty)$ such that $\{d<L\} \subset B_R$.\\
\indent The previous result, holding in the setting of a nilpotent Lie group $\mathbb{G}$, implies in particular that for a probability measure $\mu$ of the form $d\mu=Z^{-1}e^{-U(d)}dx$, where $dx$ stands for the Lebesgue measure, satisfies the $q$-Poincar\'{e} inequality provided that $\mu$ satisfies \eqref{Ubound}. The fact that the $q$-Poincar\'{e} holds on the balls with the measure $dx$ is due to Jerison; see Theorem \ref{thm.jerison}. Specifically, in the case where $\mathbb{G}=\mathbb{H}_n$, the Heisenberg group, and the measure $\mu$ is given by $d\mu=Z^{-1}e^{-ad^p}dx$, where $d$ denotes homogeneous norm associated with the Carnot-Carath\'{e}dory metric, we have that \eqref{Ubound} is satisfied, and, therefore, also the $q$-Poincar\'{e} inequality (see Theorem 2.4 \cite{HZ10}).\\
\indent Notice that in the case of $\mathcal{B}_4$, Lemma \ref{lemma,engel} implies that it would be enough to have the inclusion $\{N^{p-3}\vertiii{\cdot}^3<L\} \subset B_R$ for some $R>0$ to get the $q$-Poincar\'{e} inequality for our measure $\nu_p$. However, one can easily check that there is no such $R$ for which the previous inclusion holds, therefore one needs to proceed with a different method. 

\begin{proof}[Proof of Theorem \ref{thm.engel}:]
First notice that 
\begin{eqnarray}\label{observ-f-m}
\nu_{p}(|f-\nu_{p}f|^q)&=& \nu_{p} (|(f-m)+(m-\nu_{p}f)|^q)\nonumber\\
&\leq& \nu_{p} (\left(|f-m|+|m-\nu_{p}f| \right)^q)\nonumber\\
&\leq& \nu_{p}\Big (\big( |f-m|+\nu_{p}(|f-m|)\big)^q\Big)\,\quad (\text{since}\quad |m-\nu_{p}f|\leq\nu_{p}|f-m|)\nonumber\\
&\leq & \nu_{p} \left( 2^{q-1} \left\lbrace|f-m|^q+(\nu_{p}(|f-m|))^q\right\rbrace\right)\nonumber\\
&=&2^{q-1} \left\lbrace\nu_{p} (|f-m|^q)+\nu_{p}(|f-m|^q) \right\rbrace\nonumber\\
&=&2^q \nu_{p} (|f-m|^q)\,,
\end{eqnarray}
for any $ m\ \in \mathbb{R}$. For $R>0$ and $L>1$,
\begin{eqnarray}\label{threecases}
\nu_{p}(|f-m|)^q&=&\nu_{p} \left( |f-m|^q \one_{\{\vertiii{\cdot}^3 N^{p-3}\geq R\}} \right)+\nu_{p} \left( |f-m|^q \one_{\{\vertiii{\cdot}^3 N^{p-3}\leq R\}}\one_{\{N\leq L\}} \right)\nonumber\\
&+&\nu_{p} \left( |f-m|^q \one_{\{\vertiii{\cdot}^3 N^{p-3}\leq R\}}\one_{\{N\geq L\}} \right)\,,
\end{eqnarray}
where each term of the above sum will be treated separately.
\\

\noindent \textbf{First term of} \eqref{threecases}: Using Lemma \ref{lemma,engel} we get
\begin{eqnarray}\label{firstcase}
\nu_{p} \left( |f-m|^q \one_{\{\vertiii{\cdot}^3 N^{p-3}\geq R\}} \right)&\leq & \frac{1}{R}\nu_{p} \left( |f-m|^q N^{p-3} \vertiii{\cdot}^3 \right)\nonumber\\
&\leq & \frac{C_1}{R}\nu_{p}(|\nabla_{\mathcal{B}_4}f|^q)+\frac{D_1}{R}\nu_{p} (|f-m|^q)\,.
\end{eqnarray}
\textbf{Second term of} \eqref{threecases}: Since all homogeneous (not necessarily symmetric) norms on a group $\mathbb{G}$ are equivalent, we know that there exists a constant $C>0$ such that 
\begin{equation}\label{homog.}
C^{-1}N(x)\leq d(x) \leq C N(x)\,,\quad x \in \mathcal{B}_4\,,\nonumber
\end{equation}
where $d$ is the Carnot-Carath\'{e}odory distance and $d(x):=d(x,0)$. Thus, given $L>1$, there exists $L_1,L_2$ such that 
\[
\{N \leq L\}:=\{x \in \mathcal{B}_4\,:N(x) \leq L\} \subset B_{L_1}:=\{x \in \mathcal{B}_4\,:\, d(x)\leq L_1\} \subset \{N \leq L_2\}\,.
\]
Now, if $f \in W^{1,q}(\mathbb{R}^4)$, arguing as in Lemma \ref{lemma,engel} one can show that there exists a sequence $(f_n) \in C_{c}^{\infty}(\mathbb{R}^4)$ such that  $\left.f_n\right|_{B_{L_1}}\longrightarrow f$ in $W^{1,q}(B_{L_1})$. Then, given $n \in \mathbb{N}$, and setting 
\[
m=\frac{1}{|B_{L_1}|}\int_{B_{L_1}}f_n(x)\,dx\,,
\]
by using Theorem \ref{thm.jerison} one gets:
\begin{eqnarray*}
\nu_{p}\left(|f_n-m|^{q} \one_{\{\vertiii{\cdot}^3N^{p-3} \leq R\}}\one_{\{N\leq L\}}\right)& \leq & \frac{1}{Z}\int_{\{N\leq L\}} |f_n(x)-m|^{q}\,dx\nonumber\\
&\leq& \frac{1}{Z}\int_{\{d\leq L_1\}} |f_n(x)-m|^q\,dx\nonumber\\
&\leq& \frac{P_0(L_1)}{Z}\int_{\{d \leq L_1\}}|\nabla_{\mathcal{B}_4}f_n(x)|^q\,dx\\
& \leq & \frac{P_0(L_1)}{Z}\int_{\{N \leq L_2\}}|\nabla_{\mathcal{B}_4}f_n(x)|^q\,dx
\end{eqnarray*}
and letting $n \longrightarrow \infty$ we have
\begin{equation}\label{sec.case.engel.a}
\nu_{p}\left(|f-m|^{q} \one_{\{\vertiii{\cdot}^3N^{p-3} \leq R\}}\one_{\{N\leq L\}}\right) \leq \frac{P_0(L_1)}{Z}\int_{\{N\leq L_2\}}|\nabla_{\mathcal{B}_4}f(x)|^q\,dx\,.
\end{equation}
We now estimate the right-hand side of the above inequality as follows
\begin{equation}
    \label{sec.case.engel.b}
    \int_{\{N \leq L_2\}}|\nabla_{\mathcal{B}_4}f(x)|^q\,dx \leq e^{aL_{2}^{p}} \int_{\{N \leq L_2\}}|\nabla_{\mathcal{B}_4}f(x)|^qe^{-aN^{p}(x)}\,dx\leq e^{aL_{2}^{p}} \nu_{p}|\nabla_{\mathcal{B}_4}f|^q\,,
\end{equation}
so that, combining \eqref{sec.case.engel.a} with \eqref{sec.case.engel.b}, we finally have
\begin{equation}\label{second.case.thm.eng}
    \nu_{p}\left(|f_n-m|^{q} \one_{\{\vertiii{\cdot}^3N^{p-3} \leq R\}}\one_{\{N\leq L\}}\right) \leq P_0(L_1) e^{aL_{2}^{p}} \nu_{p}|\nabla_{\mathcal{B}_4}f|^q\,.
\end{equation}
\textbf{Third term of} \eqref{threecases}: Set $\overline{f}=f-m$ and define the set
 \[
A_{L,R}:=\{x \in \mathcal{B}_{4} : \vertiii{x}^3 \leq R, N(x) \geq L\}\,.
\]
Since $L>1$, we have
\[
\{ x \in \mathcal{B}_4 : \vertiii{\cdot}^3 N^{p-3}(x) \leq R, N(x) \geq L\} \subset A_{L,R}\,.
\]
We fix $L>R$ and choose $L$ and $R$ later.

Let now $\gamma: [0,t]\rightarrow G$ be a horizontal curve such that $|\Dot{\gamma}(s)|\leq 1$, $\vertiii{\gamma(t)\circ x}\geq R$, and there exists $C>0$ such that 
\begin{equation}
   R\leq N(\gamma(s)\circ x)\leq N(x),\quad \forall s\in[0,t].
\end{equation}
The construction of such a curve is possible at a local level; one can estimate \eqref{threecases} locally and then find the desired estimate on the whole set by standard arguments. Details are left to the reader. Hence, below, we assume for convenience that the curve is defined globally on $A_{L,R}$ and perform the argument directly.
Thus, using the notation $h:=\gamma(t)$, we have
\begin{eqnarray}\label{with_h}
\nu_{p} \left( |f-m|^q \one_{\{\vertiii{x}^3 N^{p-3} \leq R\}} \one_{\{N \geq L\}}\right)& \leq & \int_{A_{L,R}} |\overline{f}(x)|^q\,d\nu_{p}(x)\nonumber\\
& \leq & \int_{A_{L,R}^{}} |\overline{f}(x)-\overline{f}(h\circ x)|^q\,d\nu_{p}(x) \nonumber\\
&+&\int_{A_{L,R}^{}} |\overline{f}(h\circ x)|^q\,d\nu_{p}( x)\,.\label{eqn22061}
\end{eqnarray}

The first term in the last inequality can be bounded as follows
\begin{align*}
    \int_{A_{L,R}^{}} |\overline{f}(x)-\overline{f}(h\circ x)|^q\,d\nu_{p}(x)&= \int_{A_{L,R}^{}} \left|\int_0^t \frac{d}{ds}\overline{f}(\gamma(s)\circ x)\right|^q\,d\nu_{p}(x)\\
    &\leq C\int_{A_{L,R}^{}} |\nabla_{\mathcal{B}_4} f (\gamma(s)\circ x)|^q d\nu_p(x)\\
    &\leq C \int_{A_{L,R}^{}} |\nabla_{\mathcal{B}_4} f (\gamma(s)\circ x)|^q d\nu_p(\gamma(s)\circ x),\\
     &\leq C \int_{A'_{L,R}} |\nabla_{\mathcal{B}_4} f (y)|^q d\nu_p(y),\\
    &\leq  C \nu_p(|\nabla_{\mathcal{B}_4} f|^q ),
\end{align*}
where in the third last line we have used the property $N(h\circ x)^p\leq N(x)^p$, while in the second last line we applied a change of variables.

To estimate the last term in \eqref{eqn22061} first obsevrve that it as always possible to  choose  $\gamma$ in such a way so  that for a chosen element  $h=(h_1,h_2,h_3,h_4)$ satisfying $h_2\geq 2\sqrt[3]{R}$ we  have $\vertiii{h \circ x} \geq |2\sqrt[3]{R}|-|x_2| \geq \sqrt[3]{R}$  on $A_{L,R}^{}$. Then by using Lemma \ref{lemma,engel} and the inequalities above, we get 

\begin{eqnarray}\label{thirdcase}
\lefteqn{\int_{A_{L,R}^{}} |\overline{f}(h \circ x)|^q\,d\nu_{p}(x)}\nonumber\\
& \leq & \frac{1}{R R^{p-3}} \int_{A_{L,R}} |\overline{f}(h \circ x)|^q \vertiii{h \circ x}^3 N^{p-3}(h \circ x)\,d\nu_{p}(h \circ x)\nonumber\\
& \leq & \frac{1}{R^{p-2}} \nu_{p} \left(|\overline{f}|^q \vertiii{\cdot}^3 N^{p-3} \right)\nonumber\\
& \leq & \frac{C_2}{R^{p-2}} \nu_{p}(|\nabla_{\mathcal{B}_4}f|^q) +  \frac{D_2}{R^{p-2}} \nu_{p}(|f-m|^q)\,,
\end{eqnarray}
for some constants $C_2,D_2$. This completes the proof of the estimate for the third term of \eqref{threecases}. Inserting the estimates \eqref{firstcase}, \eqref{second.case.thm.eng} and \eqref{thirdcase} into \eqref{threecases},  we arrive at 
\begin{align*}
\nu_{p}(|f-m|^q) \leq& \left( \frac{C_1}{R}+P_{0}(L_1)e^{aL_{2}^{p}}+\frac{C_2}{R^{p-2}} \right) \nu_{p}(|\nabla_{\mathcal{B}_4}f|^q) \\
+&\left(\frac{D_1}{R}+\frac{D_2}{R^{p-2}} \right) \nu_{p}(|f-m|^q)\,,
\end{align*}
where $R,L$ can be taken large enough so that $\frac{D_1}{R}+\frac{D_2}{R^{p-2}}<1$ and $L>R$. Upon rearrangements the last inequality together with \eqref{observ-f-m}
  proves Theorem \ref{thm.engel}.
\end{proof}
Using Lemma \ref{lemma,engel} and a perturbation technique one can obtain the following generalisation. 
\begin{corollary}\label{cor,gen,eng}
Let $p\geq 3$, $q$ conjugate exponent of $p$ (i.e. $\frac{1}{p}+\frac{1}{q}=1$),  and $\nu_p$ be as in \eqref{prob.m.engel}. 
Let also $d\nu_{w}=\tilde{Z}^{-1}e^{-W}d\nu_{p}$ be a probability measure, where $W$ is a differentiable potential and $\tilde{Z}:=\int e^{-W}d\nu_p$ a normalization constant.
If there exist $0<\delta<\frac{1}{C(q)}$ and $\gamma_{\delta} \in (0,\infty)$ such that
\begin{equation}
\label{gen.meas}
  |\nabla_{\mathcal{B}_4}W|^{q}\leq \delta N^{p-3}\vertiii{\cdot}^3+\gamma_{\delta} ,
\end{equation} 
then the measure $\nu_{w}$ satisfies \eqref{lemma,eng,stat}. Moreover, if there exists $\tilde{C}>0$ such that $W \leq \tilde{C}N$, then $\nu_{w}$ satisfies the $q$-Poincar{\'e} inequality.
\end{corollary}

\begin{proof}
To prove the first part of Corollary \ref{cor,gen,eng} we proceed by substituting $fe^{-\frac{W}{q}}$ in the inequality \eqref{lemma,eng,stat} and get 
\begin{equation}
\label{gen,stat,1}
   \nu_{p}\left( e^{-W}|f|^q N^{p-3}\vertiii{\cdot}^3 \right) \leq C \nu_{p} \left(|\nabla_{\mathcal{B}_4}(e^{-\frac{W}{q}}f)|^q\right)+D \nu_{p} \left( |e^{-\frac{W}{q}}f|^q \right) \,,
\end{equation}
where 
\begin{eqnarray}
\label{gen,stat,2}
|\nabla_{\mathcal{B}_4}(e^{-\frac{W}{q}}f)|^q & = & \left|\left(\nabla_{\mathcal{B}_4}e^{-\frac{W}{q}}\right)|f|+e^{-\frac{W}{q}}\nabla_{\mathcal{B}_4}|f|\right|^q\nonumber\\
& \leq & \left( \frac{|\nabla_{\mathcal{B}_4}W|}{q}|e^{-\frac{W}{q}}f|+e^{-\frac{W}{q}}|\nabla_{\mathcal{B}_4}f|\right)^q\nonumber\\
& \leq & C(q) \left(|\nabla_{\mathcal{B}_4}W|^q e^{-W}|f|^q+e^{-W}|\nabla_{\mathcal{B}_4}f|^q\right)\,.
\end{eqnarray}
Substituting \eqref{gen,stat,2} in \eqref{gen,stat,1} and using \eqref{gen.meas} we get
\begin{eqnarray*}
\nu_{w} \left(f^q N^{p-3}\vertiii{\cdot}^3 \right) & \leq & CC(q) \nu_{w} \left( |\nabla_{\mathcal{B}_4}W|^q|f|^q\right) +CC(q) \nu_{w} (|\nabla_{\mathcal{B}_4}f|^q)+D\nu_{w} (|f|^q)\nonumber\\
& \leq & \delta CC(q) \nu_{w} (N^{p-3}\vertiii{\cdot}^3|f|^q)+\gamma_{\delta} CC(q)\nu_{w} (|f|^q)\nonumber\\
&+& CC(q) \nu_{w} (|\nabla_{\mathcal{B}_4}f|^q)+D\nu_{w} (|f|^q)\,,
\end{eqnarray*}
and this proves our first claim provided that $1-\delta CC(q)>0$, that is  for $\delta< C(q)^{-1}$, with $C(q)$ being a new suitable constant depending on $q$.\\
\indent Now to prove the $q$-Poincar{\'e} inequality for the measure $\nu_{w}$ we decompose $\nu_{w}(|f-m|^q)$ as in \eqref{threecases}, i.e., we write
\begin{eqnarray}\label{threecases,gen,}
\nu_{w}(|f-m|^q)&=&\nu_{w} \left( |f-m|^q \one_{\{\vertiii{\cdot}^3 N^{p-3}\geq R\}} \right)+\nu_{w} \left( |f-m|^q \one_{\{\vertiii{\cdot}^3 N^{p-3}\leq R\}}\one_{\{N\leq L\}} \right)\nonumber\\
&+&\nu_{w} \left( |f-m|^q \one_{\{\vertiii{\cdot}^3 N^{p-3}\leq R\}}\one_{\{N\geq L\}} \right)\,,
\end{eqnarray}
for some $R>0$ and $L>1$. Notice that for the first and third terms of \eqref{threecases,gen,} one can proceed as in Theorem \ref{thm.engel}. For the second term of \eqref{threecases,gen,}, arguing as in Theorem \ref{thm.engel} \eqref{sec.case.engel.a}, we get
\begin{eqnarray*}
\nu_{w}(|f-m|^q \one_{\{\vertiii{\cdot}^3N^{p-3} \leq R\}}\one_{\{N \leq L\}})& \leq & \frac{P_0(L_1)}{\tilde{Z}}\int_{\{N \leq L_2\}}|\nabla_{\mathcal{B}_4}f|^q\,dx\\
& \leq & \frac{P_0(L_1)}{\tilde{Z}}e^{aL_{2}^{p}+\tilde{C}L_2}\nu_{w} (|\nabla_{\mathcal{B}_4}f|^q)\,,
\end{eqnarray*}
since $W \leq \tilde{C}N \leq \tilde{C}L_2$ in $\{N \leq L_2\}$. This completes the proof.
\end{proof}

\begin{remark}\label{rmk.Engel}
We remark that, given a measure $\nu_p$ for which the $q$-Poincar\'{e} inequality holds, it is not in general possible to get the same property for any other measure $\nu'_p$ equivalent to $\nu_p$. Indeed it is immediate to see that, using the $q$-Poincar\'{e} inequality for $\nu_p$ together with the equivalence of the measures, we would get different probability measures on the right and on the left hand side of the inequality.
\end{remark}
\begin{remark}
Note that the $q$-Poincar{\'e} inequality in the setting of $\mathcal{B}_4$, as well as the statement of Corollary \ref{cor,gen,eng}, will still hold true if one chooses to consider the (canonical) left-invariant sub-gradient on the group instead of the right-invariant sub-gradient. This follows from Theorem \ref{thm.carnot}, that applies to the general setting of Carnot groups with filiform algebra, in the particular case where $n=3$.
\end{remark}

\section{$q$-Poincar{\'e} Inequality on filiform Carnot groups of $n$-step}\label{q-sp.gen}
In this section we prove the $q$-Poincar{\'e} inequality in the more general setting of a filiform Carnot group $\mathbb{G}_{n+1}$ of $n$-step, where $n \geq 3$, as described in Section \ref{Carnot}. The followed strategy relies on the one developed for the Engel group $\mathcal{B}_4$. 
However, here we will be working with the canonical left-invariant vector fields, a difference with respect to the analysis of $\mathcal{B}_4$ in the section above. As previously mentioned, this serves as a proof that, at least for the Engel group $\mathcal{B}_4$, Poincarè inequalities hold true  with both left and right-invariant canonical $\mathbb{G}$-grandients.\\
\indent We start by generalizing the homogeneous norm $N$ on $\mathcal{B}_4$ in the $\mathbb{G}_{n+1}$-setting. Indeed the function 
\begin{equation}
    \label{defn,norm,car}
    \tn(x):= \left(\|x\|^n+|x_{n+1}|\right)^{\frac{1}{n}}\,,
\end{equation}
where $\|x\|^n:=\sum_{j=2}^{n}\left( \aj\right)^{\fr}$, defines a homogeneous norm on $\mathbb{G}_{n+1}$. Actually, the precise generalization of the norm $N$ on $\mathcal{B}_4$ is as in \eqref{defn,norm,car} with the sum taken from $j=3$ instead of $j=2$. Let us call such norm $\bar{N}$. We could very well consider this norm in place of $\tilde{N}$ in our analysis, and the results below could be spelled out in terms of $\bar{N}$ with no modifications. Our choice to take the norm $\tilde{N}$ is just to simplify and slightly shorten the expression of some computations in Lemma \ref{lemma.N,carnot}. However, Lemma \ref{lemma.N,carnot} is still true for $\bar{N}$, and thus also the main theorem, namely Theorem \ref{thm.carnot} (i.e. the $q$-Poincaré inequality) which follows from this lemma.  For completeness, and to convince the reader that everything can be repeated with $\bar{N}$, we will show in Remark \ref{rem.gen.norm} below some key computations for this norm.

\indent Note that, in what follows, since the changes of the constants of the form $C_n$ that occur when additional terms appear do not play an essential role for our final result, for the reader's convenience we shall simply denote them $C_n$ at any appearance. Formally, one can initially choose $C_n$'s to be large enough so that they can ``absorb" any additional term that might appear. \\


\indent Recall that the generators of the lie algebra $\mathfrak{g}_{n+1}$ (the canonical left-invariant vector fields) are of the form
\[
X_1=\partial_{x_1}\,,\quad \text{and}\quad X_2=\partial_{x_2}+x_1\partial_{x_3}+\frac{x_{1}^{2}}{2!}\partial_{x_4}+\cdots+\frac{x_{1}^{n-1}}{(n-1)!}\partial_{x_{n+1}}\,.
\]
The following lemma describes the behaviour of the norm $\tn$ under the action of the canonical left-invariant operators $\subg$ and $\subl$, that is, more precisely, it shows the validity of some bounds for the sub-gradient and the sub-Laplacian of $\tn$ needed to apply our technique.
\begin{lemma}
\label{lemma.N,carnot}
Let $\tn$ be the norm in \eqref{defn,norm,car} defined on the homogeneous filiform Carnot group $\mathbb{G}_{n+1}$, $ n\geq 3$. Then, $\tn$ is smooth on $\mathbb{G}_{n+1}^{'}:=\mathbb{G}_{n+1}\setminus \{\bigcup\limits_{j=1}^{n+1}\{x: x_j=0\} \}$, and, in particular, for $x \in \mathbb{G}_{n+1}^{'}$ we have the estimates
\begin{equation}
    \label{est.subg,carnot}
    |\subg \tn(x)| \leq C_{n}^{1} \frac{\|x\|^{n-1}}{\tn^{n-1}(x)}\,,
\end{equation}
and 
\begin{equation}
    \label{est.subl,carnot}
    \subl \tn(x) \leq C_{n}^{2} \frac{\|x\|^{n-2}}{\tn(x)^{n-1}}\,,
\end{equation}
for some constants $C_{n}^{1},C_{n}^{2}>0$.
\end{lemma}
\begin{proof}
For $x \in \gp$, $x \neq 0$, we have
\begin{equation*}
    X_1(\|x\|^n+|x_{n+1}|)=sgn(x_1)\,n\,|x_1|^{\frac{n-1}{2}}\S(\aj)^{\frac{2n}{n+1}-1},
\end{equation*}
and
\begin{equation*}
\begin{split}
    X_2(\|x\|^n+|x_{n+1}|)&=n\S 
    sgn(x_j)\frac{x_1^{j-2}}{(j-1)!}(1+\delta(j-2))\,|x_j|^{\frac{n+1}{2(j-1)}-1}\\
    &\times(\aj)^{\frac{2n}{n+1}-1}+ \frac{x_1^{n-1}}{(n-1)!}sgn(x_{n+1}),
    \end{split}
\end{equation*}
where $\delta(j-2)=1$ if $j-2=0$ and $\delta(j-2)=0$ otherwise.
By using the previous computations we get
\begin{equation}
\label{estim.for.upper.carnot}
\begin{split}
X_1 \tn(x)&=\frac{1}{n\tn^{n-1}}X_1(\|x\|^n+|x_{n+1}|)\\
&=\frac{sgn(x_1)\, |x_{1}|^{\frac{n-1}{2}} \S \left(\aj \right)^{\fr-1}  }{\tn^{n-1}(x)}\,,
\end{split}
\end{equation}
and, similarly, 
 
  \begin{equation*} 
\begin{split}
X_{2}\tn(x) & = \frac{1}{n\tn^{n-1}}X_2(\|x\|^n+|x_{n+1}|)\\
&=\frac{ \S \left(\aj \right)^{\fr-1}}{n\tn^{n-1}(x)} \\
 & \times \frac{ sgn(x_j)(1+\delta(j-2))  \frac{x_{1}^{j-2}}{(j-1)!}  |x_{j}|^{\frac{n+1}{2(j-1)}-1} }{n\tn^{n-1}(x)}+\frac{ \frac{x_1^{n-1}}{(n-1)!}\,sgn(x_{n+1})}{n\tn^{n-1}(x)}.
\end{split}
\end{equation*}
Therefore we have
\[
|X_1\tn(x)| \leq  C_{n}^{1}\frac{\|x\|^{n-1}}{\tn^{n-1}(x)}\,,
\]
and 
\[
|X_{2}\tn(x)| \leq C_{n}^{1} \frac{\|x\|^{\frac{n-1}{2}}\S \|x\|^{j-2}\|x\|^{\left(\frac{n+1}{2(j-1)}-1 \right)(j-1)}+\|x\|^{n-1}}{\tn^{n-1}(x)}\leq C_{n}^{1}\frac{\|x\|^{n-1}}{\tn^{n-1}(x)}\,,
\]
since $|x_1|,|x_j|^{\frac{1}{j-1}}\leq \|x\|$ and $\aj \leq \|x\|^{\frac{n+1}{2}}$, which concludes the proof of \eqref{est.subg,carnot}.\\
\indent As regards the proof of \eqref{est.subl,carnot}, we first observe that 
\begin{equation*}
    \begin{split}
        X_1^2\tn(x)&=\frac{c_n^1 sgn(x_1)|x_1|^{\frac{n-3}{2}}  
        \S (\aj)^{\frac{2n}{n+1}-1}}{\tn(x)^{n-1}}\\
        &+\frac{c_n^2sgn(x_1)|x_1|^{n-1}\S(\aj)^{\frac{2n}{n+1}-2}}{\tn^{n-1}}\\
        &-\frac{n-1}{n}\frac{(X_1(\|x\|^n+|x_{n+1}|))^2}{\tn(x)^{2n-1}},
    \end{split}
    \end{equation*}
where $c_n^1, c_n^2$ are constants depending on $n$, and that
\begin{equation*}
    \begin{split}
        X_2^2\tn(x)&=\frac{ \S c^1_{n,j}\, sgn(x_j)\,x_1^{2(j-1)} |x_j|^{\frac{n+1}{2(j-1)}-2} (\aj)^{\frac{2n}{n+1}-1}}{\tn(x)^{n-1}}\\
        &+\frac{\S c^2_{n,j} sgn(x_j)\,x_1^{2(j-1)} |x_j|^{\frac{n+1}{2(j-1)}-1} (\aj)^{\frac{2n}{n+1}-2}}{\tn(x)^{n-1}}\\
         &-\frac{n-1}{n}\frac{(X_2(\|x\|^n+|x_{n+1}|))^2}{\tn(x)^{2n-1}},
    \end{split}
\end{equation*}
where $c_{n,j}^1, c_{n,j}^2$ are constants depending on $n$ and $j$.

Therefore, for all $x\in \mathbb{G'}_{n+1}$, we have
\begin{equation*}
    \begin{split}
  \subl \tn(x)&\leq C_n^2 \frac{\|x\|^{n-2}}{\tn^{n-1}(x)}
  -\frac{n-1}{n}\frac{(X_1(\|x\|^n+|x_{n+1}|))^2+(X_2(\|x\|^n+|x_{n+1}|))^2}{\tn(x)^{2n-1}}\\
  & \leq C_n^2 \frac{\|x\|^{n-2}}{\tn^{n-1}(x)},
    \end{split}
\end{equation*}
which concludes the proof of \eqref{est.subl,carnot}.
\end{proof}

\begin{remark}[Upper bound for $\nabla_{\mathbb{G}_{n+1}}\bar{N}$]\label{rem.gen.norm}
We show here some crucial calculations for the norm  $\bar{N}$ defined as
\begin{equation}\label{norm.Nbar}
    \bar{N}:=\left(\|x\|^n+|x_{n+1}|\right)^{\frac{1}{n}}\,,
\end{equation}
where 
$$\|x\|^n:=\sum_{j=3}^{n}\left( \aj\right)^{\fr}.$$
Note that this norm coincides exactly with $N$ when $n=3$ in $\mathbb{G}_{n+1}$.  \\
\indent We will give formulas for the quantities $X_j \bar{N}(x)$, $j=1,2$,  needed to compute the left-invariant sub-gradient and sub-Laplacian on $\mathbb{G}'_{n+1}$. By following the strategy in the proof of Lemma \ref{lemma.N,carnot}, the calculations below readily give the validity  of Lemma \ref{lemma.N,carnot} for $\bar{N}$.\\
\indent For $x \in \gp$, $x \neq 0$, we have
\begin{equation*}
    X_1(\|x\|^n+|x_{n+1}|)=sgn(x_1)\,n\,|x_1|^{\frac{n-1}{2}}\sum_{j=3}^n(\aj)^{\frac{2n}{n+1}-1},
\end{equation*}
and
\begin{equation*}
\begin{split}
    X_2(\|x\|^n+|x_{n+1}|)&=n\sum_{j=3}^n\big(sgn(x_2)|x_2|^{\frac{n+1}{2}-1}+ sgn(x_j)\frac{x_1^{j-2}}{(j-2)!}\,|x_j|^{\frac{n+1}{2(j-1)}-1}\big)\\
    &\times(\aj)^{\frac{2n}{n+1}-1}+ \frac{x_1^{n-1}}{(n-1)!}sgn(x_{n+1}).
    \end{split}
\end{equation*}
Note that comparing the formulas above with those for $\tilde{N}$, we have a slightly longer formula for $X_2\bar N$. This will not cause any issue. In fact, since
$$X_j \bar{N}(x)=\frac{1}{n \bar{N}^{n-1}}X_j(\|x\|^n+|x_{n+1}|),\quad \forall x\in\mathbb{G}'_{n+1},\quad \forall j=1,2,$$
and
\begin{align*}
    |x_1|^{\frac{n+1}{2}-1}&\lesssim \|x\|^{\frac{n-1}{2}}\\
    |x_1|^{j-2}&\lesssim \|x\|^{j-2}\\
    |x_2|^{\frac{n+1}{2}-1}&\lesssim \|x\|^{\frac{n-1}{2}}\\
    |x_j|^{\frac{n+1}{2(j-1)}-1}&\lesssim \|x\|^{\frac{n-1-2j}{2}}\\
    (\aj)^{\frac{2n}{n+1}-1}&\lesssim \|x\|^{\frac{n-1}{2}},
\end{align*}
then, arguing as in the proof of Lemma \ref{lemma.N,carnot}, we get
$$|\subg \bar{N}(x)| \leq C_{n}^{1} \frac{\|x\|^{n-1}}{\tn^{n-1}(x)},\quad \forall x\in \mathbb{G}'_{n+1},$$
that is \eqref{est.subg,carnot} for $\bar{N}$. Following similar steps one can prove \eqref{est.subl,carnot} for $\bar{N}$, and thus the desired result for $\bar{N}$. Observe finally that one can follow the proof of the main theorem below and get the $q$-Poncaré inequality on $\mathbb{G}_{n+1}$ with $\bar{N}$ in place of $\tilde{N}$.
\end{remark}

Any filiform Carnot group $\g$, $n \geq 3$, of $n$-step can be equipped with the following probability measure
\begin{equation}
    \label{prob.measure.carnot}
    \mu_p(dx):=\frac{e^{-a\tn^p(x)}}{Z}dx\,,
\end{equation}
where $p \in (1, \infty)$, $a>0$, $dx$ is the Lebesgue measure on $\mathbb{R}^{n+1}$, and $Z=\int e^{-a\tn^p(x)}\,dx$ is the normalisation constant, 
\begin{theorem}\label{thm.carnot} 
\label{Theorem 4.2}
Let $\g$, $n \geq 3$ be a filiform Carnot group of $n$-step as above. If $p \geq n$, then the measure $\mu_{p}$ as in \eqref{prob.measure.carnot} satisfies a $q$-Poincar{\'e} inequality, i.e., there exists a constant $c_0$ such that 
\begin{equation}
\label{thm.carnot.stat}
\mu_{p} (|f-\mu_{p}f|^{q}) \leq c_0 \mu_{p}(|\subg f|^{q})\,,
\end{equation}
 where $\frac{1}{p}+\frac{1}{q}=1$, for all functions $f$ for which the right hand side is finite.
\end{theorem}
\begin{remark}
Similarly as for the Engel group, we emphasize  that with suitable growth of the logarithm of the density of the measure we can obtain the $q$-Poincar{\'e} inequality with $q < 2$ which is much stronger than and implies the ordinary Poincar{\'e} inequality in $L_2$, (see e.g. Proposition 2.3 in \cite{BZ05}). Thus, as a consequence, we get a spectral gap for the Dirichlet operator defined with our probability measures and hence an exponential convergence to equilibrium in $L^2$ for the corresponding semigroup. We also note that using the same perturbation techniques we can include a class of probability measures which besides $N^p$ contain terms with slower growth at infinity.
\end{remark}

For the proof of Theorem \ref{thm.carnot}, one needs, as for the case of the Engel group $\mathcal{B}_4$, a result of the following form. 
\begin{lemma}\label{lemma,carnot} 
Let $\g$, $n \geq 3$ be a filiform Carnot group of $n$-step as above, and let $p,q$ be as in Theorem \ref{thm.carnot}. Then, for the probability measure $\mu_p$ as in \eqref{prob.measure.carnot}, there exists positive constants $C,D$ such that 
\begin{equation}\label{lemma,carnot,stat}
\mu_p (|f|^q \tn^{p-n} \vertiii{\cdot}^n) \leq C \mu_{p} (|\subg f |^{q})+D \mu_{p} (|f|^{q})\,,
\end{equation}
for any suitable $f$, where for $x \in \g$, we define $\vertiii{x}:=|x_1|$.

\end{lemma}
Bellow we give the proof of Lemma \ref{lemma,carnot} omitting some details that are similar to the proof of the corresponding lemma in the setting of $\mathcal{B}_4$ (see Lemma \ref{lemma,engel}).
\begin{proof}
We write $\gp=\cup_{j \in \mathcal{J}}C_j$, where $C_j$ are the connected components of $\gp$. We then fix $j \in \mathcal{J}$, and consider $f \in C^{\infty}(C_j)$, $ f \geq 0$. An application of the Leibniz rule gives
\[
e^{-a\tn^p}(\subg f)=\subg (fe^{-a\tn^p})+apfN^{p-1}(\subg \tn)e^{-a\tn^p}\,,
\]
so that by taking the inner product of the above quantity with $\frac{\tn^{n-1}}{\|\cdot\|^{n-1}}\subg \tn$, and integrating over $C_j$ with respect to $\mu_p$, one gets 
\begin{eqnarray}
\label{thm.carnot,eq1}
\lefteqn{\int_{C_j} \frac{\tn^{n-1}(x)}{\|x\|^{n-1}} \subg \tn(x)  \cdot\subg f(x)e^{-a\tn^p(x)}\,dx}\nonumber\\
 & = & \int_{C_j} \frac{\tn^{n-1}(x)}{\|x\|^{n-1}} \subg \tn(x)\cdot \subg\left(f(x)e^{-a\tn^p(x)}\right)\,dx\nonumber\\
&+& ap \int_{C_j} f(x) \frac{\tn^{p+n-2}(x)}{\|x\|^{n-1}} |\subg \tn|^2e^{-a\tn^p(x)}\,dx\,.
\end{eqnarray}
Applying the Cauchy-Schwartz inequality on the inner product $\subg \tn(x) \cdot \subg f(x)$ on the left-hand side of \eqref{thm.carnot,eq1}, and integrating by parts on the first term of the right-hand side of \eqref{thm.carnot,eq1}, we get
\begin{eqnarray}\label{thm.carnot,eq2}
\lefteqn{ap C_{n}^{3}\int_{C_j} f(x)\tn^{p-n}(x)\vertiii{x}^{n-1} e^{-a\tn^p(x)}\,dx}\\
& \leq & C_{n}^{1}\int_{C_j}  |\subg f(x)| e^{-a\tn^p(x)}\,dx\nonumber\\
&+& \int_{C_j} f(x) \subg \cdot \left( \frac{\tn^{n-1}(x)}{\|x\|^{n-1}}\subg \tn(x) \right) e^{-a\tn^p(x)}\,dx\,.\nonumber\\
\end{eqnarray}
where we applied \eqref{est.subg,carnot} and the following lower bound derived from \eqref{estim.for.upper.carnot}
\[
|\subg \tn(x)|> |X_1 \tn(x)| \geq C_{n}^{1}\frac{\|x\|^{\frac{n-1}{2}}\vertiii{x}^{\frac{n-1}{2}}}{N^{n-1}}\,.
\] 
Since 
\begin{eqnarray}\label{est.small.carnot1}
|X_1\|x\||  \leq  C_n^3\left( \S \aj \right)^{\frac{2}{n+1}-1}|x_1|^{\frac{n+1}{2}-1}\leq C_{n}^{3}\,,
\end{eqnarray}
and
\begin{eqnarray}\label{est.small.carnot2}
|X_2\|x\|| \leq  C_{n}^{4}  \S \!\!\left(
|x_1|^{\frac{n+1}{2}}\!+\!|x_2|^{\frac{n+1}{2}}\!+\!|x_j|^{\frac{n+1}{2(j-1)}}\!\right)^{\!\!\frac{2}{n+1}-1}  \!\!|x_1|^{j-2}|x_j|^{\frac{n+1}{2(j-1)}-1}\!\! \leq\! C_{n}^{4}\,,
\end{eqnarray}
we have, by using \eqref{est.subg,carnot} and \eqref{est.subl,carnot}, that
\begin{eqnarray}
\label{thm.carnot.eq3}
\lefteqn{\subg \cdot \left( \frac{\tn^{n-1}(x)}{\|x\|^{n-1}}\subg \tn(x)\right)}\nonumber\\
& = & \frac{\tn^{n-1}(x)}{\|x\|^{n-1}}\subl \tn(x)+(n-1) \frac{\tn^{n-2}(x)}{\|x\|^{n-1}}|\subg \tn(x)|^2\nonumber\\
& +& (1-n)\frac{\tn^{n-1}(x)}{\|x\|^n}\subg \tn(x)\cdot \subg \|x\|\nonumber\\
& \leq & C_{n} \left( \frac{\|x\|^{n-1}}{\tn(x)^n}+\frac{1}{\|x\|}\right)\,,
\end{eqnarray}
where in the previous step we used the Cauchy-Schwartz inequality as follows
\[
(1-n) \frac{\tn^{n-1}(x)}{\|x\|^n}\subg \tn(x) \cdot \subg \|x\| \leq (n-1)  \frac{\tn^{n-1}(x)}{\|x\|^n} |\subg \tn(x)| |\subg \|x\||\,.
\]
Hence, by \eqref{thm.carnot.eq3}, estimate \eqref{thm.carnot,eq2} becomes
\begin{eqnarray*}
\label{thm.carnot,eq4}
apC_{n}^{3} \mu_{p}(f \tn^{p-n}\vertiii{\cdot}^{n-1} \one_{C_j}) & \leq & C_{n}^{1} \mu_{p}(|\subg  f| \one_{C_j})\nonumber\\
&+&C_{n}\mu_{p}\left( f \left(\frac{\|\cdot\|^{n-1}}{\tn^n}+\frac{1}{\|\cdot\|}\right) \one_{C_j}\right)\,,
\end{eqnarray*}
and, after replacing $f$ with $f\|\cdot\|$, and by using \eqref{est.small.carnot1}, \eqref{est.small.carnot2} and $\|\cdot\|\geq \vertiii{\cdot}$, we obtain 
\begin{equation}\label{thm.carnot,eq5}
apC_{n}^{3}\mu_{p}(f \tn^{p-n}\vertiii{\cdot}^{n} \one_{C_j}) \leq C_{n}^{1}\mu_{p} (|\subg f|\one_{C_j})+C_n \mu_{p}(f \one_{C_j})\,.
\end{equation}
Finally, if one replaces $f$ with $f^q$, where $q$ is the conjugate exponent of $p$, then \eqref{thm.carnot,eq5} together with an application of Young's inequality allow us to estimate further as
\[
\left(apC_{n}^{3}-C_{n}^{1} \frac{q}{p}\epsilon \right)\mu_{p}(f^q\tn^{p-n}\vertiii{\cdot}^n \one_{C_j}) \leq \frac{C_{n}^{1}}{\epsilon^{q-1}}\mu_{p}(|\subg f|^q \one_{C_j})+C_n \mu_{p}(f^q \one_{C_j})\,,
\]
where the last inequality holds true for every $\epsilon>0$. Therefore, after a suitable choice of $\epsilon$, \eqref{lemma,carnot,stat} holds true for smooth, non-negative $f$ with a compact support lying in some $C_j$, with
\[
C=\frac{C_{n}^{1}}{\epsilon^{q-1}(C_{n}^{3}ap-C_{n}^{1}\frac{q}{p}\epsilon)}\,,\quad \text{and}\quad D=\frac{C_{n}}{C_{n}^{3}p-C_{n}^{1}\frac{q}{p}\epsilon}\,.
\]
An approximation argument allows \eqref{lemma,carnot,stat} to be valid for any suitable $f$. The proof is complete.
\end{proof}
For the proof of Theorem \ref{thm.carnot} one argues similarly to Theorem \ref{thm.engel}.
\begin{proof}[Proof of Theorem \ref{thm.carnot}:]
Since $\mu_p|f-\mu_{p}f|^q \leq 2^q \mu_{p}|f-m|^q$, for any $m>0$, it is enough to prove an upper bound of the form \eqref{thm.carnot.stat} for each term of the below decomposition 
\begin{eqnarray}\label{threecases,carnot}
\mu_{p}|f-m|^q&=&\mu_{p} \left( |f-m|^q \one_{\{\vertiii{\cdot}^n \tn^{p-n}\geq R\}} \right)+\mu_{p} \left( |f-m|^q \one_{\{\vertiii{\cdot}^n \tn^{p-n}\leq R\}}\one_{\{\tn\leq L\}} \right)\nonumber\\
&+&\mu_{p} \left( |f-m|^q \one_{\{\vertiii{\cdot}^n \tn^{p-n}\leq R\}}\one_{\{\tn\geq L\}} \right)\,,
\end{eqnarray}
for some $R>0$, $L>1$. The first and second terms of \eqref{threecases,carnot} can  be treated by simply adapting the strategy followed in Theorem \ref{thm.engel} in the general setting considered here. For the third term of \eqref{threecases,carnot} one proceeds as follows. \\
\indent Set $\overline{f}=f-m$ and define the set 
\[
A_{L,R}:= \{ x \in \g : \vertiii{x}^n \leq R\,,\tn(x) \geq L \}\,.
\]
We fix $L>R$ and choose $L$ and $R$ later.

As in the proof of Theorem \ref{thm.engel} we consider a  horizontal curve $\gamma: [0,t]\rightarrow \g$  such that $\gamma(0)=0$, $|\Dot{\gamma}(s)|\leq 1$, $\vertiii{\gamma(t)\circ x}\geq R$, and there exists $C>0$ such that 
\begin{equation}
   R\leq N(\gamma(s)\circ x)\leq N(x),\quad \forall s\in[0,t].
\end{equation}
Then, since 
\[
\{ x \in \g : \vertiii{x}^n \tn^{p-n}(x) \leq R\,,\tn(x) \geq L\}\subset A_{L,R}\,,
\]
we can estimate as 
\begin{eqnarray*}
\label{3case,carnot}
\mu_{p} \left(|\overline{f}|^q \one_{\{\vertiii{\cdot}^n \tn^{p-n} \leq R\}}\one_{\{N \geq L\}} \right) & \leq & \int_{A_{L,R}} |\overline{f}(x)|^q\,d\nu_{p}(x)\nonumber\\
& \leq & \int_{A_{L,R}^{}} |\overline{f}(x)-\overline{f}(h\circ x)|^q\,d\nu_{p}(x) \nonumber\\
&+&\int_{A_{L,R}^{}} |\overline{f}(h\circ x)|^q\,d\nu_{p}( x),\,
\end{eqnarray*}
and bound the two terms in the last inequality as in the proof of Theorem \ref{thm.engel}.
In particular we will have
\begin{align}
\int_{A_{L,R}^{}} |\overline{f}(x)-\overline{f}(h\circ x)|^q\,d\nu_{p}(x)\leq C \nu_p(|\nabla_{\g}f|^q),\\
\end{align}
and, by choosing $h=(h_1,\ldots,h_n)$ such that $h_1\geq 2R^{\frac{1}{n}}$ so that $\vertiii{x \circ h} \geq R^{\frac{1}{n}} $, we get 
\begin{eqnarray}
\label{last,carnot}
\lefteqn{\int_{A_{L,R}^{'}}|\overline{f}(x \circ h)|^q\,d\mu_{p}(x \circ h)}\nonumber\\
& \leq & \frac{1}{RL^{p-n}}\int_{A_{L,R}^{'}} |\overline{f}(x \circ h)|^q \vertiii{x \circ h}^n \tn^{p-n}(x \circ h) d\mu_{p}(x \circ h)\nonumber\\
& \leq & \frac{C}{R^{p-n+1}}\mu_{p}|\subg f|^q+\frac{D}{R^{p-n+1}}\mu_{p}|f-m|^q\,
\end{eqnarray}
by Lemma \ref{lemma,carnot}. Inequality \eqref{last,carnot}, as well as the other terms of \eqref{threecases,carnot}, should now be handled as in the case of $\mathcal{B}_4$ (see Theorem \ref{thm.engel}), that is by choosing $R$ and $L$  big enough with $L>R$. This completes the proof of Theorem \ref{thm.carnot}.
\end{proof}

\begin{remark}
Note that the key point in the proof of Theorem \ref{thm.carnot} is Lemma \ref{lemma,carnot}, the latter following from Lemma \ref{lemma.N,carnot}. Hence, by Remark \ref{rem.gen.norm}, we get that Theorem \ref{thm.carnot} holds with $\tilde{N}$ replaced by $\bar{N}$.
\end{remark}

As in the case of the Engel group $\mathcal{B}_4$, one can extend the family of measures satisfying \eqref{thm.carnot.stat} in our general setting by using a perturbation technique of \cite{HZ10}. 
\begin{corollary}\label{cor,gen,carnot}
Let $p\geq 3$, $q$ be the conjugate exponent of $p$ (i.e. $\frac{1}{p}+\frac{1}{q}=1$),  and $\mu_p$ be as in \eqref{prob.measure.carnot}. 
Let also $d\mu_{w}=\tilde{Z}^{-1}e^{-W}d\mu_{p}$ be a probability measure, where $W$ is a differentiable potential  and $\tilde{Z}:=\int e^{-W}d\mu_p$ a normalization constant.
If there exist $0<\delta<\frac{1}{C(q)}$ and $\gamma_{\delta} \in (0,\infty)$ such that
\begin{equation}
\label{gen.meas,carnot}
  |\subg W|^{q}\leq \delta \tn^{p-n}\vertiii{\cdot}^n+\gamma_{\delta},
\end{equation}
then, the measure $\mu_{w}$ satisfies \eqref{lemma,carnot,stat}. Moreover, if there exists $\tilde{C}>0$ such that $W \leq \tilde{C}\tn$, then $\mu_{w}$ satisfies the $q$-Poincar{\'e} inequality.
\end{corollary}

\begin{proof}
To prove the first part of Corollary \ref{cor,gen,carnot} we proceed by plugging the function $f e^{-\frac{W}{q}}$ in the inequality \eqref{lemma,carnot,stat} to get 
\begin{equation}
\label{gen,stat,1,carnot}
   \mu_{p}\left( e^{-W}|f|^q \tn^{p-n}\vertiii{\cdot}^n \right) \leq C \mu_{p} \left(|\subg(e^{-\frac{W}{q}}f)|^q\right)+D \mu_{p} \left( |e^{-\frac{W}{q}}f|^q \right) \,,
\end{equation}
where 
\begin{eqnarray}
\label{gen,stat,2,carnot}
|\subg(e^{-\frac{W}{q}}f)|^q & = & \left|\left(\subg e^{-\frac{W}{q}}\right)|f|+e^{-\frac{W}{q}}\subg|f|\right|^q\nonumber\\
& \leq & \left( \frac{|\subg W|}{q}|e^{-\frac{W}{q}}f|+e^{-\frac{W}{q}}|\subg f|\right)^q\nonumber\\
& \leq & C(q) \left(|\subg W|^q e^{-W}|f|^q+e^{-W}|\subg f|^q\right)\,.
\end{eqnarray}
Inserting \eqref{gen,stat,2,carnot} into \eqref{gen,stat,1,carnot} and using \eqref{gen.meas,carnot} we arrive at
\begin{eqnarray*}
\mu_{w} \left(|f|^q \tn^{p-n}\vertiii{\cdot}^n \right) & \leq & CC(q) \mu_{w} \left( |\subg W|^q|f|^q\right) +CC(q) \mu_{w} (|\subg f|^q)+D\mu_{w} (|f|^q)\nonumber\\
& \leq & \delta CC(q) \mu_{w} (\tn^{p-n}\vertiii{\cdot}^n|f|^q)+\gamma_{\delta} CC(q)\mu_{w} (|f|^q)\nonumber\\
&+& CC(q) \mu_{w} (|\subg f|^q)+D\mu_{w} (|f|^q)\,,
\end{eqnarray*}
and this proves our first claim provided that $1-\delta CC(q)>0$, that is  for $\delta< C(q)^{-1}$, with $C(q)$ being a new suitable constant depending on $q$.\\
\indent Now to prove the $q$-Poincar{\'e} for the generalised measure, we decompose $\mu_{w}(|f-m|^q)$ as in \eqref{threecases,carnot}, i.e., we write
\begin{eqnarray}\label{threecases,gen,carnot}
\mu_{w}|f-m|^q&=&\mu_{w} \left( |f-m|^q \one_{\{\vertiii{\cdot}^n \tn^{p-n}\geq R\}} \right)+\mu_{w} \left( |f-m|^q \one_{\{\vertiii{\cdot}^n \tn^{p-n}\leq R\}}\one_{\{\tn\leq L\}} \right)\nonumber\\
&+&\mu_{w} \left( |f-m|^q \one_{\{\vertiii{\cdot}^n \tn^{p-n}\leq R\}}\one_{\{\tn\geq L\}} \right)\,,
\end{eqnarray}
for some $R>0$ and $L>1$. Notice that for the first and third terms of \eqref{threecases,gen,carnot} one can proceed as in Theorem \ref{thm.carnot}. For the second term of \eqref{threecases,gen,carnot},
arguing as in Theorem \ref{thm.engel} we conclude that
\begin{eqnarray*}
\mu_{w}(|f-m|^q \one_{\{\vertiii{\cdot}^n \tn^{p-n} \leq R\}}\one_{\{\tn \leq L\}})& \leq & \frac{P_0(L_1)}{\tilde{Z}}\int_{\tn \leq L_2}|\subg f|^q\,dx\\
& \leq & \frac{P_0(L_1)}{\tilde{Z}}e^{aL_{2}^{p}+\tilde{C}L_2}\mu_{w} (|\subg f|^q)\,,
\end{eqnarray*}
since $W \leq \tilde{C}\tn \leq \tilde{C}L_2$ in $\{\tn \leq L_2\}$. This finally shows the second part of Corollary \ref{cor,gen,carnot} and completes the proof.
\end{proof}

\begin{remark}
As for the particular case of the Engel group (see Remark \ref{rmk.Engel}), in this more general setting we have, once again, that it is not possible to pass from the $q$-Poincar\'{e} inequality for $\mu_p$ to the same inequality for any measure $\mu'_p$ equivalent to $\mu_p$. Therefore, as explained in Remark \ref{rmk.Engel}, the equivalence of the measures is not enough to pass the property from one measure to the other.
\end{remark}

\end{document}